\documentclass[12pt]{amsart}
\usepackage{amscd}
\usepackage{amssymb}
\usepackage{graphicx}
\usepackage{rotate,epic,eepic,epsf,epsfig}

\allowdisplaybreaks[1]

\newtheorem{thm}{Theorem}[section]
\newtheorem{rem}[thm]{Remark}

\newtheorem{cor}[thm]{Corollary}
\newtheorem{lem}[thm]{Lemma}
\newtheorem{prop}[thm]{Proposition}

\newtheorem{defn}[thm]{Definition}
\theoremstyle{definition}
\newtheorem{exple}[thm]{Example}

\numberwithin{equation}{section}

\DeclareFixedFont{\petitefonte}{\encodingdefault}%
{\familydefault}{\seriesdefault}{\shapedefault}{5pt}

\newcommand{\N}{\mathbb N}
\newcommand{\n}{\noindent}

\newcommand{\resumename}{R\'esum\'e}

\begin{document}

\title[Diamond cone for $\mathfrak{so}(2n+1)$]{Diamond module for the Lie algebra $\mathfrak{so}(2n+1,\mathbb C)$}

\author[B. Agrebaoui, D. Arnal and A. Ben Hassine]{Boujema\^a Agrebaoui, Didier Arnal and Abdelkader Ben Hassine}
\address{Universit\'e de Sfax\\
Facult\'e des Sciences\\
D\'epartement de Mathematiques\\
Route de Soukra, km 3,5, B.P. 1171\\
3000 Sfax, Tunisie.} \email{B.Agreba@fss.rnu.tn}
\address{
Institut de Math\'ematiques de Bourgogne\\
UMR CNRS 5584\\
Universit\'e de Bourgogne\\
U.F.R. Sciences et Techniques
B.P. 47870\\
F-21078 Dijon Cedex\\France} \email{Didier.Arnal@u-bourgogne.fr}
\address{Universit\'e de Sfax\\
Facult\'e des Sciences\\
D\'epartement de Mathematiques\\
Route de Soukra, km 3,5, B.P. 1171\\
3000 Sfax, Tunisie.} \email{benhassine.abdelkader@yahoo.fr}

\thanks{B. Agrebaoui, and A. Ben Hassine thank the Institut de Math\'ematiques de Bourgogne for his hospitality during his stays in Dijon, D. Arnal thanks the University of Sfax for its support and hospitality during his visits in Tunisia.}


\begin{abstract}
The diamond cone is a combinatorial description for a basis of an indecomposable module for the nilpotent factor $\mathfrak n$ of a semi simple Lie algebra. After N. J. Wildberger who introduced this notion, this description was achevied for $\mathfrak{sl}(n)$, the rank $2$ semi-simple Lie algebras and $\mathfrak{sp}(2n)$.\\

In the present work, we generalize these constructions to the Lie algebras $\mathfrak{so}(2n+1)$. The orthogonal semistandard Young tableaux were defined by M. Kashiwara and T. Nakashima, they form a basis for the shape algebra of $\mathfrak{so}(2n+1)$. Defining the notion of orthogonal quasistandard Young tableaux, we prove these tableaux give a basis for the diamond module for $\mathfrak{so}(2n+1)$.\\
\end{abstract}


\maketitle

\vspace{.20cm}


\section{Introduction}


\

Let $\mathfrak g$ be a semisimple Lie algebra. The simple (finite dimensional) $\mathfrak g$-modules are characterized by their highest weight $\lambda$, each of them contains an unique (up to constant) vector $v_\lambda$ with weight $\lambda$, the $\mathfrak g$-action on $v_\lambda$ generates the corresponding simple module. The direct sum of all these modules is a natural algebra, the shape algebra of $\mathfrak g$.\\

Consider now the nilpotent factor $\mathfrak n$ in the Isawasa decomposition the Lie algebra $\mathfrak g$. It is natural to study nilpotent finite dimensional $\mathfrak n$-modules. They are generally indecomposable, if the module is generated by the action on an unique vector $v$, we say this module is monogenic. Each of the monogenic nilpotent module is a quotient of a well determined simple $\mathfrak g$-module (viewed as a $\mathfrak n$-module). The natural object corresponding to the shape algebra is now the diamond module, union of all these maximal monogenic modules.\\

We report on a program to construct explicit combinatorial model for a basis in the diamond module, called the diamond cone. Such a description is given in the case of $\mathfrak{sl}(n)$ by D. Arnal, N. Bel Barka, and N. J. Wildberger in \cite{ABW}, in the case of $\mathfrak{sp}(2n)$, by D. Arnal and O. Khlifi \cite{AK} and, in the case of rank two semisimple Lie algebras, by B. Agrebaoui, D. Arnal and O. Khlifi \cite{AAK}.\\

Let us first recall the $\mathfrak{sl}(n)$ case, which is the simplest one. In this case, the shape algebra (the direct sum of all finite dimensional irreducible representations) admits a well known basis given by semistandard Young tableaux $T$, if we restrict ourselves to the semistandard tableaux with shape $\lambda$, we get a basis for the irreducible module $\mathbb S^\lambda$, with highest weight characterized by the shape, and still denoted $\lambda$. There is a notion of quasistandard tableau. Denote $QS^\lambda$ the subset of quasistandard tableaux in the set $SS^\lambda$ of semistandard tableaux with shape $\lambda$.\\

For any tableau $T$ in $SS^\lambda$, which is not in $QS^\lambda$, there is procedure, based on the usual jeu de taquin ($jdt$) which transforms $T$ in a new tableau $p(T)$, which is quasistandard, with a shape $\mu<\lambda$. Putting $p(T)=T$ if $T$ is quasistandard, it is possible to prove that the map:
$$
p:~SS^\lambda~\longrightarrow~\bigsqcup_{\mu\leq\lambda}QS^\mu,
$$
is a bijective map. In other words, we have an indexation of a basis for the module $\mathbb S^\lambda$, which is well adapted to the description of the $\mathfrak n$ indecomposable module $\mathbb S^\lambda|_{\mathfrak n}$. Indeed, any maximal monogenic $\mathfrak n$ submodule in $\mathbb S^\lambda$ is the subspace generated by $\bigsqcup_{\mu\leq\nu}QS^\mu$ for some $\nu\leq\lambda$ (see \cite{ABW} for details).\\

The situation is very similar for the $\mathfrak{sp}(2n)$ case. A basis for the simple modules $\mathbb S^{\langle\lambda\rangle}$ is given by the set $SS^{\langle\lambda\rangle}$ of symplectic semistandard Young tableaux with shape $\lambda$, in \cite{AK}, the notion of symplectic quasistandard Young tableaux is given, let $QS^{\langle\lambda\rangle}$ be the set of such tableaux with shape $\lambda$, using the symplectic jeu de taquin ($sjdt$) defined by J. T. Sheats (\cite{S}), define a bijective map:
$$
p:~SS^{\langle\lambda\rangle}~\longrightarrow~\bigsqcup_{\mu\leq\lambda}QS^{\langle\mu\rangle}.
$$
With this map, we get a basis for the module $\mathbb S^{\langle\lambda\rangle}|_{\mathfrak n}$, well adapted with its stratification.\\

The goal of this paper is to realize the same program for the $\mathfrak{so}(2n+1)$ case. First we recall the definition of semistandard Young tableaux for $\mathfrak{so}(2n+1)$, given by Kashiwara and Nakashima (see \cite{KN}, see also the presentation given by Lecouvey in \cite{L}). In this construction, Lecouvey defines the split of an orthogonal Young tableau. An orthogonal semistandard Young tableau with shape $\lambda$ is a tableau $T$ such that its split $spl(T)$ is symplectic semistandard: $spl(T)\in SS^{\langle2\lambda\rangle}$. Unfortunately, this choice is not convenient for our purpose of quasistandardness, therefore we modify the presentation of orthogonal semistandard tableaux, and the splitting procedure, in order to get a new map $dble$ and say that a tableau $T$ is orthogonal semistandard ($T\in SS^{[\lambda]}$) if and only if $dble(T)\in SS^{\langle2\lambda\rangle}$. We prove the equivalence between our construction and the Lecouvey's one, by proving that the $spl$ and the $dble$ of respective semistandard Young tableaux form the same subset in $SS^{\langle2\lambda\rangle}$.\\

We are now able to define orthogonal quasistandard tableau by the same method as for the symplectic case. Denote $QS^{[\lambda]}$ the set of such tableaux, with shape $\lambda$, we want to prove that the orthogonal jeu de taquin ($ojdt$), defined by Lecouvey allows to build a bijective map:
$$
p:~SS^{[\lambda]}~\longrightarrow~\bigsqcup_{\mu\leq\lambda}QS^{[\mu]}.
$$
The orthogonal jeu de taquin on an orthogonal tableau $T$ is defined by using the symplectic jeu de taquin on the split form $spl(T)$ of $T$, thus it is well defined on our notion of semistandard orthogonal Young tableau. Unfortunately, Lecouvey does not give a rule for this jeu de taquin, directly on the tableau $T$, therefore we first give such an explicit and direct expression of the action of the jeu de taquin on $T$ itself, at least in the case we consider, {\sl i.e.} when the jeu de taquin motion is horizontal.\\

Thanks to this expression, we can define the map $p$, as a `maximal' (in a sense explained below) action of the orthogonal jeu de taquin $ojdt$, we compute the inverse mapping of $ojdt$ and prove that $p$ is a bijective map.\\ 

With this map, we get a basis for the module $\mathbb S^{[\lambda]}|_{\mathfrak n}$, well adapted with its stratification.\\


\section{Semistandard and quasistandard Young tableaux for $\mathfrak{sl}(n)$}


\subsection{Semistandard Young tableaux}

\

The theory of finite dimensional representations of semisimple Lie algebras is well known an very explicit. In the classical cases, we have a natural representation on a complex space $V$. For $\mathfrak{sl}(n)$, $V=\mathbb C^n$. We first consider simple modules in the tensor product $\otimes^\ell V$. We recover with these modules all the simple $\mathfrak{sl}(n)$ modules. The key to understanding the decomposition of $\otimes^\ell V$ is the Schur-Weyl duality.\\

Let $\mathbb C[S_\ell]$ be the group algebra of the symmetric group $S_\ell$. It is a semisimple algebra, its  simple components are indexed by the set of partitions $\mathcal{P}(\ell)$ of $\ell$ (weakly decreasing sequence of positive integers $\lambda=(\lambda_1\geq\lambda_2\geq\cdots\geq \lambda_n)$ whose sum is $\ell$.)
$$
\mathbb C[S_\ell]=\oplus_{\lambda\in \mathcal{P}}B_\lambda.
$$

Let $S_\ell$ acts on the right side on $\otimes^\ell V$ by permutation:
$$
{\rho_\ell(\sigma)}(v_1\otimes\cdots\otimes v_\ell)=(v_1\otimes\cdots\otimes v_\ell)^\sigma=v_{\sigma(1)}\otimes\cdots\otimes v_{\sigma(\ell)}.
$$
The Shur-Weyl duality theorem is the fact that the commutant of the natural representation of $\mathfrak{gl}(V)$ in ${\otimes^\ell}V$ is exactly $\rho_\ell(\mathbb C[S_\ell])$ (see \cite{GW}).\\

Therefore, we have the following decomposition of ${\otimes^\ell}V$ as ($\mathfrak{gl}(V)-S_\ell$) simple modules:
$$
{\otimes^\ell}V=\bigoplus_{\lambda\in\mathcal{P}(n,\ell)} \mathbb{S}_{\lambda}\otimes B_\lambda,
$$
where $\mathcal{P}(n,\ell)$ is the set of partitions of $\ell$, with lenght $n$: $\{\lambda=(\lambda_1\geq\lambda_2\geq\cdots\geq \lambda_n),~~\sum\lambda_j=\ell\}$. A Young diagram of shape $\lambda\in\mathcal{P}(n,\ell)$ is a tableau of empty boxes, with $\lambda_j-\lambda_{j+1}$ columns with height $j$ ($1\leq j\leq n-1$) and $\lambda_n$ columns with height $n$.

A standard Young tableau of shape $\lambda$ is the filling of the corresponding Young diagram, with positive integers in $\{1,\ldots,\ell\}$, such that the entries are strictly increasing along rows and columns. The set of standard Young tableaux gives a basis for $B_\lambda$. Similarly, a semistandard Young tableau for $\mathfrak{gl}(n)$ with shape $\lambda$ is the filling of the corresponding Young diagram, with positive integers in $\{1,\ldots,n\}$, such that the entries are strictly increasing along columns and weakly increasing along rows. The set of semistandard Young tableaux for $\mathfrak{gl}(n)$ gives a basis for $\mathbb S_\lambda$. Explicitely, for any pair $(T,S)$ of a semistandard tableau $T$ for $\mathfrak{gl}(n)$ and a standard tableau $S$, we associate the tensor product:
$$
\rho_\lambda(Y_\lambda)(e_{t_1}\otimes\dots\otimes e_{t_\ell}),
$$
where  $(e_1,\dots,e_n)$ is the canonical basis of $V$, the entries $t_k$ of $T$ are indexed by the entry of the corresponding boxes in $S$, and $Y_\lambda$ is the Young symmetrizer: the element in $\mathbb C[S_\lambda]$ giving the projection on $B_\lambda$.\\

To get a realization of the shape algebra as $\mathfrak{gl}(n)$-module, we choose for each $\lambda$ a particular standard tableau, namely the filling of the corresponding Young diagram row by row from the top to the bottom and from the left to the right. With this choice the highest weight vector $v_\lambda$ in $\mathbb S^\lambda$ is associated to the `trivial' semistandard tableau, for which the boxes in the $i^{\text{th}}$ row ares filled by the integer $i$ and $v_\lambda$ is:
$$
v_\lambda=(e_1)^{\lambda_1-\lambda_2}\cdot(e_1\wedge e_2)^{\lambda_2-\lambda_3}\cdot\ldots\cdot(e_1\wedge\dots\wedge e_n)^{\lambda_n}.
$$ 

Finally, the restriction $\mathbb S^\lambda$ of $\mathbb S_\lambda$ to $\mathfrak{sl}(n)$ is simple, and two such restrictions $\mathbb S^{\lambda}$ and $\mathbb S^{\lambda'}$ coincide if and only if $\lambda_j=\lambda'_j$ for any $j<n$. We thus only consider partitions $\lambda$ with $\lambda_n=0$. Recall that the usual ordering on weights is then: $\mu\leq\lambda$ if and only if $\mu_j-\mu_{j+1}\leq\lambda_j-\lambda_{j+1}$ for any $j$, $1\leq j\leq n-1$.\\ 

Since the group $G=SL(n,\mathbb C)$ is a classic, connected and simply connected Lie group, we can realize the shape algebra as the space of affine functions on the quotient $\overline{N}\setminus G$, where $\overline{N}$ is the Lie subgroup corresponding to the nilpotent factor $\overline{\mathfrak n}$ opposite to $\mathfrak n$ (see \cite{GW}).\\

\subsection{Quasistandard Young tableaux and jeu de taquin}

\

Denote $\mathbb{S}^\bullet_{red}$ the diamond module for $\mathfrak{sl}(n)$. We can realize explicitly this module as the quotient of the shape algebra by the ideal generated by the elements $v_\lambda-1$ (for any $\lambda$), or as the space of polynomial functions on $N$.\\

As a $\mathfrak n$-module, $\mathbb{S}^\bullet_{red}$ is indecomposable, this is the union of the modules $\mathbb{S}^\lambda|_{\mathfrak n}$, with a natural layering:
$$
\mu\leq \lambda \Longleftrightarrow \mathbb{S}^\mu|_{\mathfrak n} \subset \mathbb{S}^\lambda|_{\mathfrak n}.
$$
Indeed, in the shape algebra, $v_\nu\cdot v_\mu=v_{\nu+\mu}$, thus in the quotient, the diamond module, $\mathbb S^\mu=v_\nu\cdot\mathbb S^\mu\subset \mathbb S^{\nu+\mu}$.

To get a combinatorial basis in the diamond module, it is necessary to suppress any trivial semistandard tableau, and even any semistandard tableau containing a trivial tableau with a shape $\mu\leq\lambda$. Therefore, we put (see \cite{ABW})

\begin{defn}

\

Let $T=(t_{ij})$ be a semistandard tableau, with shape $\lambda$. If the top of the first column (the $s$ first boxes) is a trivial tableau, if $T$ contains a column with height $s$, and if, for all $j$ for which these entries exist, the relation  $t_{s(j+1)}<t_{(s+1)j}$ holds, we say that $T$ is not quasistandard at the level $s$: $T\in NQS_s^\lambda$.

If there is no $s$ for which $T$ is in $NQS_s^\lambda$, we say that $T$ is quasistandard. The set of quasistandard tableaux with shape $\lambda$ is denoted $QS^\lambda$.\\
\end{defn}

The principal result in \cite{ABW} is: the quasi-standard tableaux form a basis of the diamond module. This can be proved by using the jeu de taquin ($jdt$). Let us present now this operation due to Schutzenberger.\\

Let $\mu\leq \lambda$ be two shapes, we let $Y(\mu)$, the Young diagram with shape $\mu$, as a subdiagram placed in the left-top corner of $Y(\lambda)$, the associated Young diagram to $\lambda$. An interior corner of $Y(\mu)$ is a box $(x,y)$ of $Y(\mu)$ such that, immediately in the right and immediately below to this box, there is no box of $Y(\mu)$. An exterior corner of $Y(\lambda)$ is an empty box  $(x',y')$ which we can add to $Y(\lambda)$ so that  $Y(\lambda)\cup\{(x',y')\}$ still is a Young diagram.\\

Let us leave $Y(\mu)$ empty inside $Y(\lambda)$ and fill in the skew tableau $Y(\lambda\setminus\mu)$ by integers $t_{ij}\leq n$ in a semistandard way:
For all  $i$ and all $j$, $t_{ij}<t_{(i+1)j}$ and $t_{ij}\leq t_{i(j+1)}$, if the boxes are in $Y(\lambda\setminus \mu)$.

We choose an interior corner of $Y(\mu)$ and we identify it by a star: $\boxed{\star}$. We obtain a pointed skew-tableau $T:=T(\lambda\setminus \mu)$. For example
$$
\begin{tabular}{|c|c|c|} \hline &$2$& \multicolumn{1}{|c|}{$4$}
\\ 
\hline
$\star$&$3$&\multicolumn{1}{|c|}{$5$ }\\
\cline{1-3}
$4$&$6$\\
\cline{1-2}
$5$&$7$\\
\cline{1-2}
\end{tabular}
$$

The jeu de taquin is a way to move the $\boxed{\star}$ in $T(\lambda\setminus \mu)$. After a number of moving, the tableau $T$ becomes a tableau $T'$
in which the $\star$ is in the $(i,j)$ box. The rules of the jeu de taquin is as follows:
\begin{itemize}
\item[1-] If the box $(i,j+1)$ exists and either the box $(i+1,j)$ does not exist or $t_{(i+1)j}>t_{i(j+1)}$, then we push $\boxed{\star}$ to the right,
{\sl i.e.}, we replace $T'$ by the tableau $T"$ where we put $\boxed{t_{i(j+1)}}$ in $(i,j)$, and $\boxed{\star}$ in $(i,j+1)$, the other entries in $T'$ being unchanged in $T"$.\\

\item[2-] If the box $(i+1,j)$ exists and either the box $(i,j+1)$ does not exist or $t_{(i+1)j}\leq t_{i(j+1)}$, then we push $\boxed{\star}$
downward, {\sl i.e.}, we replace $T'$ by the tableau $T"$ where we put $\boxed{t_{(i+1)j}}$ in $(i,j)$, andt $\boxed{\star}$in $(i+1,j)$, the other entries ib $T'$ being unchanged in $T"$.\\

\item[3-] If the boxes $(i+1,j)$ and $(i,j+1)$ do not exist, we remove the $\boxed{\star}$. The box $(i,j)$ is no longer a box of $T"$, but the tableau consisting of boxes of $T"$ and of the box $(i,j)$ is a Young tableau: the $(i,j)$ box is an exterior corner of $T"$.\\
\end{itemize}

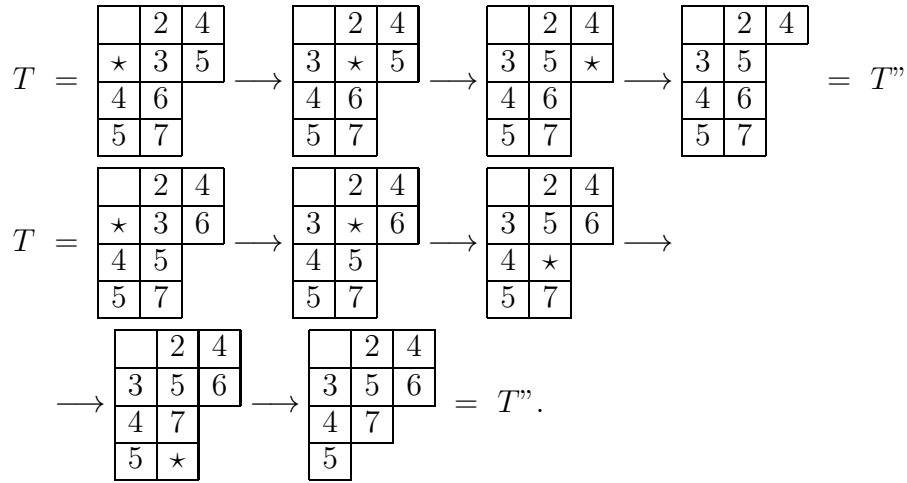
\begin{exple}

\

$$\aligned
T~&=~{\begin{tabular}{|c|c|c|}
\hline
 & $2$ & \multicolumn{1}{|c|}{$4$}\\
\hline
$\star$ & $3$&  \multicolumn{1}{|c|}{$5$ }\\
\cline{1-3}
$4$ & $6$\\
\cline{1-2}
$5$ & $7$\\
\cline{1-2}
\end{tabular}}\longrightarrow {\begin{tabular}{|c|c|c|}
\hline
 & $2$ & \multicolumn{1}{|c|}{$4$}\\
\hline
$3$ & $\star$ & \multicolumn{1}{|c|}{$5$ }\\
\cline{1-3}
$4$ & $6$\\
\cline{1-2}
$5$ & $7$\\
\cline{1-2}
\end{tabular}} \longrightarrow {\begin{tabular}{|c|c|c|}
\hline
 & $2$ & \multicolumn{1}{|c|}{$4$}\\
\hline
$3$ & \multicolumn{1}{|c|}{$5$} & $\star$\\
\cline{1-3}
$4$ & $6$\\
\cline{1-2}
$5$ & $7$\\
\cline{1-2}
\end{tabular}}\longrightarrow {\begin{tabular}{|c|c|c|}
\hline
 & $2$ & \multicolumn{1}{|c|}{$4$}\\
\hline
$3$ & \multicolumn{1}{|c|}{$5$ }\\
\cline{1-2}
$4$ & $6$\\
\cline{1-2}
$5$ & $7$\\
\cline{1-2}
\end{tabular}~=~T"}\\
T~&=~{\begin{tabular}{|c|c|c|}
\hline
 & $2$ & \multicolumn{1}{|c|}{$4$}\\
\hline
$\star$ & $3$ & \multicolumn{1}{|c|}{$6$ }\\
\cline{1-3}
$4$ & $5$\\
\cline{1-2}
$5$ & $7$\\
\cline{1-2}
\end{tabular}} \longrightarrow {\begin{tabular}{|c|c|c|}
\hline
 & $2$ & \multicolumn{1}{|c|}{$4$}\\
\hline
$3$ & $\star$ & \multicolumn{1}{|c|}{$6$ }\\
\cline{1-3}
$4$ & $5$\\
\cline{1-2}
$5$ & $7$\\
\cline{1-2}
\end{tabular}}
\longrightarrow {\begin{tabular}{|c|c|c|}
\hline
 & $2$ & \multicolumn{1}{|c|}{$4$}\\
\hline
$3$ & $5$ & \multicolumn{1}{|c|}{$6$}\\
\cline{1-3}
$4$ & $\star$\\
\cline{1-2}
$5$ & $7$\\
\cline{1-2}
\end{tabular}} \longrightarrow\\
&\longrightarrow {\begin{tabular}{|c|c|c|}
\hline
 & $2$ & \multicolumn{1}{|c|}{$4$}\\
\hline
$3$ & $5$ & \multicolumn{1}{|c|}{$6$}\\
\cline{1-3}
$4$ & $7$\\
\cline{1-2}
$5$ & $\star$\\
\cline{1-2}
\end{tabular}}
\longrightarrow {\begin{tabular}{|c|c|c|}
\hline & $2$ & \multicolumn{1}{|c|}{$4$}\\
\hline
$3$ & $5$ & \multicolumn{1}{|c|}{$6$}\\
\cline{1-3}
$4$ & $7$\\
\cline{1-2}
$5$\\
\cline{1-1}
\end{tabular}}~=~T".\endaligned
$$
\end{exple}

Let us call $S"$ the empty Young diagram obtained after removing the pointed box and let $\mu"$ be the shape of $S"$. The tableau $T"\setminus S"$ is still semistandard. If $(i,j)$ is the interior pointed corner of $S$ and $(i",j")$ the exterior pointed corner of $T"$, we set $(T"\setminus S",(i",j"))=jdt(T\setminus S,(i,j))$.

Let us explain now the inverse map: $(jdt)^{-1}$. Let $T\setminus S$ be a skew semistandard tableau of shape $\lambda\setminus\mu$, consider the smallest rectangle containing $T\setminus S$, then return this rectangle by performing a central symmetry, and replace each of the entries $t_{ij}$ of the obtained skew tableau by the entries $n+1-t_{ij}$ and $\star$ by $\star$. The obtained tableau $T'\setminus S'=\sigma(T\setminus S)$ is still a skew semistandard tableau.
If we pointed an exterior corner of $T$, the box $\boxed{\star}$ now is in an interior corner of $S'$, and  reciprocally. Then
$$
jdt^{-1}(T"\setminus S",(i",j"))=\sigma\circ jdt\circ \sigma(T"\setminus S",(i",j")).
$$
For instance, the above applied jeu de taquin is reversed, if $n=7$, as follows:
$$\aligned
(T,(4,2))&={\begin{tabular}{|c|c|c|}
\hline
 & $2$ & \multicolumn{1}{|c|}{$4$}\\
\hline
$3$ & $5$ & \multicolumn{1}{|c|}{$6$ }\\
\cline{1-3}
$4$ & $7$\\
\cline{1-2}
$5$ & $\star$\\
\cline{1-2}
\end{tabular}}&\hfill
\sigma(T,(4,2))&={\begin{tabular}{|c|c|c|}
\hline
 & $\star$ & \multicolumn{1}{|c|}{$3$}\\
\hline
 & $1$ & \multicolumn{1}{|c|}{$4$}\\
\cline{1-3}
$2$ & $3$ & \multicolumn{1}{|c|}{$5$}\\
\cline{1-3}
$4$ & $6$\\
\cline{1-2}
\end{tabular}}\\
jdt\circ \sigma(T,(4,2))&={\begin{tabular}{|c|c|c|}
\hline
 & $1$ & \multicolumn{1}{|c|}{$3$}\\
\hline
 & $3$ & \multicolumn{1}{|c|}{$4$}\\
\cline{1-3}
$2$ & $5$ & \multicolumn{1}{|c|}{$\star$}\\
\cline{1-3}
$4$ & $6$\\
\cline{1-2}
\end{tabular}}&\hfill
\sigma\circ jdt\circ \sigma(T,(4,2))&=\begin{tabular}{|c|c|c|}
\hline
 & $2$ & \multicolumn{1}{|c|}{$4$}\\
\hline
$\star$ & $3$ & \multicolumn{1}{|c|}{$6$}\\
\cline{1-3}
$4$ & $5$\\
\cline{1-2}
$5$ & $7$\\
\cline{1-2}
\end{tabular}
\endaligned
$$

The jeu de taquin is thus a bijective map:
$$\aligned
jdt~:~&\bigcup_{\lambda\setminus\mu}SS(\lambda\setminus\mu)\times\{\text{interior corners in }\mu\}~\longrightarrow\\
&\longrightarrow~\bigcup_{\lambda"\setminus\mu"}SS(\lambda"\setminus\mu")\times\{\text{exterior corners for }\lambda"\}.
\endaligned
$$

Let us now consider a non quasistandard tableau $T=(t_{ij})$ with shape $\lambda$ and let $s$ be the largest integer such that $T$ is not quasistandard in $s$.
The $s$ top entries of its first column are $1,2,\dots,s$. We call $S$ the empty tableau with only one column with height $s$: the shape of $S$ is $\mu=(1,\dots,1,0,\dots,0)$. We consider the pointed skew tableau $U$ whose entries in $\lambda\setminus \mu$ are the entries of $T$ and the pointed box is in the only interior corner in $S$. 

We apply the jeu de taquin. The pointed box moves always to the right and leaves the diagram at the end of the last column of height $s$.

The row $s$ has just been shifted by one box to the left. We obtain a skew tableau with $s-1$ empty boxes in its first column. We fill in these boxes with $1,\dots,s-1$. If $s>1$, the obtained tableau $T"$ is semistandard, not quasistandfard in $s-1$, and may be in $s$, but it is quasistandard in all $t>s$.\\

This procedure can therefore be repeated and finally we get a quasistandard tableau $T'=p(T)$.\\

It is easy to check that this procedure realize a bijection between the set of semistandard tableaux $SS^\lambda$ with shape $\lambda$ and the union
$\sqcup_{\mu \leq\lambda}~QS^{\mu}$ of sets of quasistandard tableaux with shape $\mu\leq \lambda$.
$$
SS^\lambda \longleftrightarrow \sqcup_{\mu \leq\lambda}~QS^{\mu}.
$$

Considering the quotient map from the shape algebra $\mathbb S^\bullet$ to the diamond module $\mathbb S^\bullet_{red}$, we see the restriction of this quotient map to each $\mathbb S^\lambda$ is one-to-one. Thus we can choose only the vectors associated to quasistandard tableaux, to get a basis for the quotient.

\begin{thm} $($\cite{ABW}$)$
The set $QS^\bullet$ of quasistandard tableaux is the diamond cone, {\sl i.e.} a basis of the diamond module ${\mathbb S}^{\bullet}_{red}$, which describes the stratification of this indecomposable $\mathfrak n$-module.\\
\end{thm}


\section{The symplectic case}

\subsection{Symplectic Lie algebra and its representations}

\

We let $V=\mathbb{C}^{2n}$ be the $2n$ dimensional vector space with basis $(e_1,\dots,e_n, e_{\overline{n}},\dots,e_{\overline{1}})$ and equipped with a symplectic form
$$
\Omega=\sum_{i=1}^n e^\star_i \wedge e^\star_{\overline{i}}.
$$
If $M$ is a $n\times n$ matrix, we denote $^sM$ the image of $M$ under the symmetry with respect to its second diagonal. Then the symplectic (simple) Lie algebra associated to $\Omega$ can be realized as the set of matrices:
$$
\mathfrak{sp}(2n)=\Big\{
\left(
\begin{array}{cc}
M & V\\
U & -^sM\\
\end{array}\right);~M,~U,~V\in Mat(n),U= ^sU,V=^sV
~\Big\}
$$

A Cartan subalgebra $\mathfrak h$ of $\mathfrak{sp}(2n)$ consists of diagonal matrices
$$
H=diag(\kappa_1,\dots,\kappa_n,-\kappa_n,\dots,-\kappa_1).
$$
We let $\theta_j(H)=\kappa_j$ and choose the following simple roots system:
$$
\Delta=\{\alpha_i=\theta_i-\theta_{i+1},~~i=1, 2,\dots,n-1,~\alpha_n=2\theta_n\}.
$$
For this choice  $\mathfrak n=\sum_{\alpha>0}\mathfrak g^\alpha$ is the subalgebra of strictly upper triangular matrices in $\mathfrak{sp}(2n)$. Define the fundamental weights as $\omega_k=\theta_1+\dots+\theta_k,~1\leq k\leq n$. The fundamental module $\mathbb S^{\langle\omega_k\rangle}$ is realized as the kernel of the contraction $\varphi_k:\wedge^k V\longrightarrow \wedge^{k-2}V$ (with the convention $\wedge^{-1}V=0$) defined by: 
$$
{\varphi_k}(v_1\wedge\dots\wedge v_k)=\sum_{i<j}\Omega(v_i,v_j)(-1)^{i+j-1} v_1\wedge \dots \wedge \widehat{v_i}\wedge \dots \wedge \widehat{v_j} \wedge \dots \wedge v_k.
$$
Then the set $\Lambda$ of positive dominant weights is $\Lambda =\{\lambda=\sum_{k=1}^n a_k\omega_k,~ a_k\in \mathbb N\}$. For $\lambda\in\Lambda$, the irreducible module $\mathbb S^{\langle\lambda\rangle}$ is realized as a submodule of $\otimes V$, more precisely as the simple submodule in
$$
Sym^{a_1}(\mathbb S^{\langle\omega_1\rangle})\otimes Sym^{a_2}(\mathbb S^{\langle\omega_2\rangle})\otimes\dots
\otimes Sym^{a_n}(\mathbb S^{\langle\omega_n\rangle}),
$$
with highest weight $\lambda$.\\

Since all these modules are in $\otimes V$, as before, we can describe a combinatorial basis for $\mathbb S^{\langle\lambda\rangle}$, by using semistandard Young tableaux with shape $\lambda$. However, we have to select, among the usual semistandard tableaux, some of them, called symplectic semistandard tableaux. The set of such tableaux will be dentoted $SS^{\langle\lambda\rangle}$ (for details see \cite{FH,dC,KN}). We present such a choice in the two next sections.\\

\subsection{Subset in the left or the right side}

\

Let $X=\{e_1<e_2<\dots<e_p\}$ be a subset of $\{1,\dots,n\}$, denote $Y=\{1,\dots,n\}\setminus X$.

\begin{defn}

\

Let  $J$ be a subset of $X$. A subset $I$ of $X$ is said to be in the left (resp. right) side of $X$ if:
\begin{itemize}
\item[i.] $\# I=\# J$,
\item[ii.]  $I\cap J=\emptyset$,
\item[iii.] if $J\neq\emptyset$, and $J=\{y_1<y_2<\dots<y_s\}$, then $I=\{x_1<x_2<\dots<x_s\}$ and $x_i<y_i$ (resp. $y_i<x_i$) for all $i$, $1\leq i\leq s$.\\
\end{itemize}
\end{defn}

Denote $\mathrm{L}(J)$ the set of all subsets of $X$ in the left side of $J$. For instance, if
$X=[1,10]$,
$$
\mathrm{L}(\emptyset)=\{\emptyset\},\quad \mathrm{L}(\{1,3\})=\emptyset,\quad \mathrm{L}(\{2,6\})=\big\{\{1,3\},\{1,4\},\{1,5\}\big\}.
$$

Denote $\Gamma_Y$ (or $\Gamma^n_Y$) the set of subsets $J$ in $X$, such that $\mathrm{L}(J)\neq \emptyset$.\\

\begin{lem}\label{leftpart}

\

Let $J$ be in $\Gamma_Y$. Then there exists a unique largest subset, denoted $\gamma_Y(J)$, in the left side of $J$, such that if $\gamma_Y(J)=\{x_1<\dots<x_s\}$, then for all $I'=\{x'_1<\dots<x'_s\}$, in $\mathrm{L}(J)$, the relation $x'_i\leq x_i$ holds for every $i$ ($1\leq i\leq s$).\\

Let $J=\{y_1<\dots<y_s\}$ be a non empty set in $\Gamma_Y$ and $I=\gamma_Y(J)=\{x_1,\dots,x_s\}$. Set $Z=X\setminus (I\cup J)$. Let $t\in Z$. If there exists $i$ such that $t<y_i$ then $t<x_i$.\\
\end{lem}

\begin{proof} If $J=\emptyset$, then $L(J)=\{\emptyset\}$, and $\gamma_Y(\emptyset)=\emptyset$.\\

Let us now suppose $J=\{y_1<\dots<y_s\}$ is not empty. We define by induction the elements $x_i$ in $X$, as follows:
\begin{itemize}
\item[] $x_s=\sup\{t\in X\setminus J,~~t<y_s\}$,
\item[] $x_i=\sup\{t\in X\setminus J,~~t<x_{i+1},~~t<y_s\}$ ($1\leq i\leq s-1$).\\
\end{itemize}

It is easy to prove that the $x_i$ do exist, and if $I'=\{x'_1<\dots<x'_s\}$ is in $\mathrm{L}(J)$, then, by induction, $x'_i\leq x_i$ for all $i$.\\

This implies the unicity of the subset $\gamma_Y(J)=\{x_s>\dots>x_1\}$.\\

Suppose the second assertion wrong, and $x_i\leq t$. Since $t$ is not in $I$, this means $x_i<t$. Let $k$ be the largest index such that $x_k<t$ ($i\leq k$).\\

If $k=s$, this gives $t<y_i\leq y_s$, and $t\in X\setminus J$, then $x_s< t\leq x_s=\sup\{u\in X\setminus J,~u<y_s\}$, which is impossible.\\

If $k<s$, this gives $t<y_i\leq y_k$ and $t< x_{k+1}$, then $x_k<t\leq\sup\{u\in X\setminus J,~u<y_k,~u<x_{k+1}\}$, which is also impossible.\\

This proves the lemma.\\
\end{proof}

Of course the same properties are holding on the right side of a subset $I$ in $X$. Denote $\mathrm{R}(I)$ the family of all subsets in the right side of $I$, say that $I$ is in $\Delta_Y$ if $\mathrm{R}(I)$ is not empty. Remark that for $I$ in $\Delta_Y$, there exists in $R(I)$ a smallest subset denoted $\delta_Y(I)=\delta(I)$. If moreover $I=\{x_1<\dots<x_s\}$ is non empty, denote $\delta_Y(I)=\{y_1<\dots< y_s\}$.

For any $J'=\{y'_1<\dots<y'_s\}$ in $R(J)$, have $y'_i\geq y_i$ for all $i$ ($1\leq i\leq s$). Set $Z=X\setminus (I\cup J)$, let $t\in Z$. If there exists $i$ such that $t>x_i$ then $t>y_i$.\\

\subsection{Semistandard and quasistandard symplectic tableaux}

\

Consider the ordering $1<2<\dots<n<\bar{n}<\dots<\bar{1}$ and let  $A,~D$ subsets in $\{1,\dots,n\}$ such that $k=\sharp A+ \sharp D \leq n$.
Set $I=A \cap D=\{i_1,\dots,i_r\}$. Let us say that the column
$$
\begin{array}{c}
A\\
\overline{D}
\end{array}~=~\renewcommand{\arraystretch}{0.7}\begin{array}{|c|}
\hline
p_1\\
\hline
\vdots\\
\hline
p_s\\
\hline
\overline{q_t}\\
\hline
\vdots\\
\hline
\overline{q_1}\\
\hline
\end{array}
$$
is a symplectic semistandard column if $I$ is in $\Delta_{A\cup D}$.\\

To any symplectic semistandard column, we associate a two columns tableau, the double of this column. Put first:
$$
J=\delta_{A\cup D}(I),~~~~ B=(A\backslash I)\cup J,~~~~~~~ C=(D\backslash I)\cup J.
$$
Remark that, knowing $B$ and $C$, we have $J=B\cap C$ is in $\mathrm{L}(B\cup C)$ and $I=\gamma_{B\cup C}(J)$. Denote the symplectic column:
$$
\begin{array}{c}
A\\
\overline{D}
\end{array}= f(A,D)= g(B,C).
$$

The double of $\begin{array}{c}
A\\
\overline{D}
\end{array}$ is by definition the tableau
$$
dble\big(\renewcommand{\arraystretch}{0.7}\begin{array}{c}
A\\
\overline{D}
\end{array}\big)= \renewcommand{\arraystretch}{0.7}\begin{array}{cc}
A & B\\
\overline{C} & \overline{D}
\end{array}.
$$
It is a semistandard Young tableau for the chosen ordering $1<2<\dots<n<\overline{n}<\dots<\overline{1}$.\\

\begin{defn}

\

Let $T$ be a tableau of shape $\lambda$ consisting of semistandard columns. The tableau $dble(T)$ is obtained by juxtaposing the doubles of all the columns of $T$.

We say that $T$ is symplectic semistandard (or semistandard for $\mathfrak{sp}(2n)$) if $dble(T)$ is semistandard (for $\mathfrak{sl}(2n)$).\\
\end{defn}

The set of symplectic semistandard Young tableaux of shape $\lambda$ is a basis for the simple $\mathfrak{sp}(2n)$ module $\mathbb{S}^{\langle\lambda\rangle}$.
Let us denote by $v_\lambda$ its highest weight.

In a recent paper by D. Arnal and O. Khlifi (see \cite{AK}) the following two algebras are studied:
the shape algebra
$$
\mathbb{S}^{\langle\bullet\rangle}=\bigoplus_{\lambda\in\Lambda}~\mathbb{S}^{\langle\lambda\rangle}.
$$
and the reduced shape algebra (the diamond module):
$$
\mathbb S^{\langle \bullet\rangle}_{red}=\mathbb S^{\langle \bullet \rangle}\big{/}<v_\lambda-1,~\lambda\in \Lambda>.
$$

The first algebra has for basis the set of symplectic semistandard tableaux $SS^{\langle\bullet\rangle}$ while the second algebra has for basis the set of  symplectic quasistandard tableaux $QS^{\langle\bullet\rangle}$ defined as follows:\\

\begin{defn}

\

Let $T$ be a symplectic semistandard tableau. We say that $T$ is a symplectic quasistandard tableau if $dble(T)$  is quasistandard (for $\mathfrak{sl}(2n)$).\\
\end{defn}

We note that a symplectic semistandard Young tableau can be quasistandard for $\mathfrak{sl}(2n)$, but not its double. For example
$$
T=\begin{tabular}{|c|c|}
\hline
\raisebox{-2pt}{$1$} & \raisebox{-2pt}{$2$}\\
\hline
\raisebox{-2pt}{$2$} & \raisebox{-2pt}{$\overline 2$}\\
\cline{1-2}
\raisebox{-2pt}{$\overline 2$}\\
\cline{1-1}
\end{tabular}~~ \Longrightarrow ~~dble(T)=\begin{tabular}{|c|c|c|c|}
\hline
\raisebox{-2pt}{$1$} & \raisebox{-2pt}{$1$} & \raisebox{-2pt}{$2$} & \raisebox{-2pt}{$3$}\\
\hline
\raisebox{-2pt}{$2$} & \raisebox{-2pt}{$3$} & \raisebox{-2pt}{$\overline{3}$} & \raisebox{-2pt}{$\overline{2}$}\\
\cline{1-4}
\raisebox{-2pt}{$\overline{3}$} & \raisebox{-2pt}{$\overline{2}$}\\
\cline{1-2}
\end{tabular}
$$
$T$ is quasistandard but $dble(T)$ is not quasistandard. The set of all symplectic quasistandard tableaux with shape $\lambda$ is denoted $QS^{\langle\lambda\rangle}$.\\

Let us say that a symplectic semistandard tableau $T$ is not quasistandard at the level $s$ and denote $T\in NQS_s$ if $dble(T)$ is not quasistandard at the level $s$. 

In the paper \cite{S} the symplectic {\it jeu de taquin} ($sjdt$) is defined on a skew symplectic semistandard tableau, by using its double. Especially, in the case where the $\star$ moves to the right, along the row $s$ in $T$, the motion is the following:
\begin{itemize}
\item[] Suppose the row $s$ contains the entries $\star~a'$, then the left column $f(A,D)=g(B,C)$ becomes $g(B\cup\{a'\},C)$, the right column $f(A',D')$ becomes $f(A'\setminus\{a'\},D')$,
\item[] Suppose the row $s$ contains the entries $\star~\overline{d'}$, and the row $s$ in $dble(T)$ contains $\star~\overline{c'}$, then, in $T$, the left column $f(A,D)$ becomes $f(A,D\cup \{c'\})$, the right column $f(A',D')=g(B',C')$ becomes $g(B',C'\setminus\{c'\})$.\\
\end{itemize}

Using this symplectic jeu de taquin, it is possible to prove, like in the $\mathfrak{sl}(n)$ case, that the set of symplectic quasistandard tableaux is a basis for the reduced shape algebra that respect its structure of indecomposable $\mathfrak n$ module (see \cite{AK}).\\

We shall now follow the same strategy in the $\mathfrak{so}(2n+1)$ case.\\



\section{Orthogonal semistandard Young tableaux}

\subsection{$\mathfrak{so}(2n+1)$ and its positive dominant weights}
\

Let $\mathcal{B}_n=\{i, \overline{i}, 1\leq i\leq n\}\cup \{0\}$ be an ordered set with the ordering given by:
$$
1<2<\ldots<n<0<\overline{n}<\ldots<\overline{2}<\overline{1}.
$$
For any $a$, $b$ in a totally ordered set $E$, denote $[a,b]=\{x\in E,~a\leq x\leq b\}$ for instance, in $\mathcal B_n$, $[1,n]=\{1,2,\ldots,n\}$.

Put $\overline{\overline{i}}=i$ and $\overline{0}=0$. Let $V=\mathbb C^{2n+1}$ with basis $(e_1,\ldots,e_n,e_0,e_{\overline{n}},\ldots,e_{\overline{1}})$ indexed by  $\mathcal{B}_n$.

The odd dimensional orthogonal algebra $\mathfrak{g}=\mathfrak{so}(2n+1)$ is the Lie algebra given by the matrices antisymmetric with respect to the non degenerated symmetric bilinear form $Q=\langle~,~\rangle$ defined by 
$$
\langle e_i,e_{\overline{j}}\rangle=\delta_{ij},\quad\forall ~i,~j\in\mathcal{B}_n.
$$

The matrix of $Q$ is:
$$
S=
\left(\begin{array}{ccccc}
0 &0&\ldots & 0 & 1 \\
0 &0&\ldots & 1 & 0 \\
\vdots & \vdots & \cdots &  \vdots&  \vdots\\
0 &1 & \ldots &0&  0\\
1 & 0 & \ldots&0 &  0
\end{array}
\right).
$$
Denote $A\mapsto~^sA$ the symmetry with respect to the second diagonal, thus $\mathfrak g$ is the set of all $(2n+1)\times(2n+1)$ matrices $X$ so that $^sX=-X$, or:
$$
X=\left(\begin{array}{ccc}
A & u & B \\
-~^sx & 0 & -~^su \\
 C & x & -~^sA \\
\end{array}
\right)
$$
where $A$ is a $(n\times n)$-matrix, $B$ and $C$ are $(n\times n)$-matrices, such that $^sB=-B$, $^sC=-C$, $x$ and $u$ are $(n\times 1)$-matrices and, if $u$ is a column matrix, $^su=(u_{n1},\ldots,u_{11})$. The Lie algebra $\mathfrak g$ is a simple Lie algebra of type $(B_n)$.\\

Denote $E_{ij}$ the usual $n\times n$ matrix with unique non vanishing entry 1 at the row $i$ and the column $j$, and $E_i$ the column with unique non vanishing entry 1 at the row $i$, we get the following basis for $\mathfrak g$:
$$\aligned
H_i&=\left(\begin{matrix}E_{ii}&0&0\\
0&0&0\\
0&0&-^sE_{ii}\end{matrix}\right)~~(1\leq i\leq n),~~&X_{ij}&=\left(\begin{matrix}E_{ij}&0&0\\
0&0&0\\
0&0&-^sE_{ij}\end{matrix}\right)~~(1\leq i\neq j\leq n),\\
Y_{ij}&=\left(\begin{matrix}0&0&(E_{ij}-^sE_{ij})\\
0&0&0\\
0&0&0\end{matrix}\right)~~(i+j\leq n), ~~&Z_{ij}&=\left(\begin{matrix}0&0&0\\
0&0&0\\
(E_{ij}-^sE_{ij})&0&0\end{matrix}\right)~~(i+j\leq n),\\
U_i&=\left(\begin{matrix}0&E_i&0\\
0&0&-^sE_i\\
0&0&0\end{matrix}\right)~~(1\leq i\leq n), ~~&X_i&=\left(\begin{matrix}0&0&0\\
-^sE_i&0&0\\
0&E_i&0\end{matrix}\right)~~(1\leq i\leq n).
\endaligned
$$

The set of diagonal matrices
$$
H=\sum_{i=1}^n\kappa_iH_i,
$$
is a Cartan subalgebra $\mathfrak{h}$ of $\mathfrak g$. The dual space $\mathfrak{h}^*$ has for basis the $n$ forms $\epsilon_j$ where $\epsilon_j(H)=\kappa_j$.

The roots and the root spaces of $\mathfrak g$ are given by the commutation relations:
$$\aligned
&[H,X_{ij}]=(\epsilon_i-\epsilon_j)(H)X_{ij},\\
&[H,Y_{ij}]=(\epsilon_i+\epsilon_{n+1-j})(H)Y_{ij},\\
&[H,Z_{ij}]=-(\epsilon_{n+1+i}+\epsilon_j)(H)Z_{ij},\\
&[H,U_i]=\epsilon_i(H)U_i,\\
&[H,X_i]=-\epsilon_{n+1-i}(H)X_i.
\endaligned
$$

The root system  is thus $\pm\epsilon_i\pm\epsilon_j$ ($1\leq i<j\leq n$) and $\pm\epsilon_i$ ($1\leq i\leq n$). We choose the simple roots system
$$
\Phi=\{\epsilon_1-\epsilon_2,\ldots, \epsilon_{n-1}-\epsilon_n,~\epsilon_n\}.
$$
Then the positive roots are $\epsilon_i-\epsilon_j$, $\epsilon_i+\epsilon_j$ ($1\leq i<j\leq n$), and $\epsilon_i$, ($1\leq i\leq n$). The nilpotent factor $\mathfrak n$ in the Iwasawa decompsition of $\mathfrak g$ is the sum of the corresponding root spaces. It is the set of upper triangular matrices in $\mathfrak g$ or the space generated by the matrices $X_{ij}$, $Y_{ij}$, for $1\leq i<j\leq n$, and the $U_i$, with $1\leq i\leq n$.\\


The weight lattice of $\mathfrak{so}(2n+1)$ is generated by $\epsilon_1,~\epsilon_2,~\ldots, \epsilon_{n-1},~\epsilon_n$ together with the further weight $\frac{1}{2}(\varepsilon_1+\ldots+\varepsilon_n)$.

The Weyl chamber is $\mathcal{W}=\{\sum a_i\epsilon_i,~a_1\geq a_2\geq\ldots\geq a_n\geq 0\}$. The edges of the Weyl chamber are thus the rays generated by the vectors $\epsilon_1,~\epsilon_1+\epsilon_2,\ldots, \epsilon_1+\ldots +\epsilon_{n-1}$ and $\epsilon_1+\ldots+\epsilon_n$. For $\mathfrak g$, the intersection of the weight lattice with the closed Weyl cone is the free semigroup generated by the following fundamental weights: 
$$
\omega_1=\epsilon_1,~\omega_2=\epsilon_1+\epsilon_2,\ldots,~\omega_{n-1}=\epsilon_1+\ldots+\epsilon_{n-1},~\omega_n=\frac{1}{2}(\epsilon_1+\ldots+\epsilon_n).
$$
Any weight $\lambda$ in the Weyl chamber can be written: $\lambda=\sum_{i=1}^n a_i\omega_i$ with $a_i$ a natural number ($a_i\in \N$). Denote $\mathbb S^{[\lambda]}$ the corresponding simple module.\\

\subsection{Irreducible representations of $\mathfrak{so}(2n+1)$}

\

The construction of the fundamental modules $\mathbb S^{[\omega_r]}$ is explicitly presented in the excellent book \cite{FH}, by W. Fulton and J. Harris.\\

First, for $r=1,\ldots, n$, the natural antisymmetric tensor representation $\wedge^rV$ is an irreducible highest weight representations of $\mathfrak{so}(2n+1)$, with highest weight $\omega_r$ for $r<n$ and $2\omega_n$ for $r=n$. The vectors $e_{i_1}\wedge\cdots\wedge e_{i_r}$ ($1\leq i_1<\ldots<i_r\leq \overline{1}$), form a basis of $\wedge^rV$. Describe now the $\mathfrak{so}(2n+1)$-action on these vectors.\\

Recall that the standard action of $\mathfrak{so}(2n+1)$ on $V$ is given by the matrix form of the element $X\in\mathfrak g$. Especially, the Chevalley generators act as follows:
$$
X_{i,i+1}\cdot e_{i+1}=e_i,~ X_{i,i+1}\cdot e_{\overline{i}}=-e_{\overline{i+1}}\mbox{ for } 1\leq i<n \hbox{ and } U_n\cdot e_0=e_n,~ U_n\cdot e_{\overline{n}}=-e_0,
$$
(the other relations vansih) and:
$$
X_{i+1,i}\cdot e_i=e_{i+1},~ X_{i+1,i}\cdot e_{\overline{i+1}}=-e_{\overline{i}}\mbox{ for } 1\leq i<n \hbox{ and } -X_1\cdot e_n=e_0,~-X_1\cdot e_0=-e_{\overline{n}}.
$$
The action of $\mathfrak{so}(2n+1)$ on $\wedge^rV$ is the canonical one:
$$
X\cdot(e_{i_1}\wedge\cdots\wedge e_{i_r})=(X\cdot e_{i_1})\wedge\cdots\wedge e_{i_r}+\ldots+ e_{i_1}\wedge\cdots\wedge(X\cdot e_{i_r}).
$$
In particular, every $H\in\mathfrak{h}$ acts diagonnaly:
$$\aligned
H\cdot &(e_{i_1}\wedge\ldots\wedge e_{i_k}(\wedge e_0)\wedge e_{\overline{j_1}}\wedge\ldots\wedge e_{\overline{j_s}})=\\
&\hskip 1cm=(\epsilon_{i_1}+\ldots+\epsilon_{i_k}-\epsilon_{j_1}-\ldots-\epsilon_{j_{r-k}})(H)e_{i_1}\wedge\cdots\wedge e_{i_k}(\wedge e_0)\wedge e_{\overline{j_1}}\wedge\ldots\wedge e_{\overline{j_{r-k}}}.
\endaligned
$$
Hence the set of weights of the representation is
$$
\{(\epsilon_{i_1}+\ldots+\epsilon_{i_k})-(\epsilon_{j_1}+\ldots+\epsilon_{j_{r-k}}),~1\leq i_1<\ldots<i_k\leq n,~1\leq j_{r-k}<\ldots<j_1\leq n\}.
$$
The highest weight is $\omega_r=\epsilon_1+\ldots+\epsilon_r$.


There is still one fundamental representation to describe: $\mathbb S^{[\omega_n]}$.\\

\begin{defn}

\

The finite dimensional irreducible representation with the highest weight $\omega_n$ is called the spin representation and denoted by $V_{sp}$.
\end{defn}

This last fundamental representation however is more mysterious. The fundamental weight $\omega_n$ cannot be a weight of any tensor power of the
standard representation,it cannot be a submodule in $\otimes V$. We first describe directly this representation:

We index a basis for $V_{sp}$ as what we call spin column: 

They are the columns $\mathfrak C$ of height $n$ with strictly increasing entries in $[1,n]\cup[\overline{n},\overline{1}]$, such that for all $i\in\mathcal{B}_n$, $i$ and $\overline{i}$ do not appear simultaneously in $\mathfrak C$. Denote ${\mathfrak C}=\begin{array}{c}A\\ \overline{D}\end{array}_{sp}$ such a column, with $A,~D\subset[1,n]$, $\#A+\#D=n$, and $A\cap D=\emptyset$ (to simplify notations, we omit to draw the boxes). The number of such columns (the dimension of $V_{sp}$) is $2^n$. The action of $\mathfrak{so}(2n+1)$ on $V_{sp}$ is given in terms of Chevalley generators as follows:
$$
X_{i,i+1}\cdot \begin{array}{c} 
\vdots \\
i+1\\
\vdots \\
\overline{i}\\
\vdots \\
\end{array}_{sp}=\frac{1}{\sqrt{2}}\begin{array}{c} 
\vdots \\
i\\
\vdots \\
\overline{i+1}\\
\vdots \\
\end{array}_{sp}\hbox{ if } 1\leq i< n,~~ U_n\cdot \begin{array}{c} 
\vdots \\
\overline{n}\\
\vdots \\
\end{array}_{sp}=\frac{1}{\sqrt{2}}\begin{array}{c} 
\vdots \\
n\\
\vdots \\
\end{array}_{sp}
$$
(the other actions vanish). And
$$
X_{i+1,i}\cdot\begin{array}{c} 
\vdots \\
i\\
\vdots \\
\overline{i+1}\\
\vdots \\
\end{array}_{sp}=\frac{1}{\sqrt{2}}\begin{array}{c} 
\vdots \\
i+1\\
\vdots \\
\overline{i}\\
\vdots \\
\end{array}_{sp}\hbox{ if } 1\leq i< n,~ -X_1\cdot \begin{array}{c} 
\vdots \\
n\\
\vdots \\
\end{array}_{sp}=\frac{1}{\sqrt{2}}\begin{array}{c} 
\vdots \\
\overline{n}\\
\vdots \\
\end{array}_{sp}
$$
(the other actions vanish).

The weight of each column ${\mathfrak C}=\begin{array}{c}A\\ \overline{D}\end{array}_{sp}$ is $\frac{1}{2}(\sum_{i\in A}\epsilon_i-\sum_{i\in D}\epsilon_i)$. Therefore, $\begin{array}{c} 
1\\
2\\
\vdots \\
n\\
\end{array}_{sp}$
is the highest weight vector with weight $\omega_n=\frac{1}{2}(\varepsilon_1+\ldots+\varepsilon_n)$.\\

In fact, it turns out that this last fundamental representation does not come from a representation of the group $SO(2n+1,\mathbb C)$. The point here is that
this group is not simply connected, so there are Lie algebra homomorphism on $\mathfrak{so}(2n+1)$ which do not integrate to group homomorphisms on $SO(2n + 1,\mathbb C)$.
Correspondingly, the simply connected group with Lie algebra $\mathfrak{so}(2n+1)$ is called the spin group $Spin(2n + 1,\mathbb C)$. It turns out that this spin group is an extension of $SO(2n + 1,\mathbb C)$
with kernel $Z_2 = \pm 1$, i.e. there is a surjective homomorphism $Spin(2n + 1,\mathbb C)\rightarrow
SO(2n + 1,\mathbb C)$ whose kernel consists of two elements.\\

One can construct both the spin representation $V_{sp}$ and the spin group $Spin(2n + 1,\mathbb C)$ by using the Clifford algebra
$Cl(2n + 1,\mathbb C)$ of $V$, details can be founded in \cite{FH}, Chapter 20. Remark that we have:
$$
V_{sp}\otimes V_{sp}=\oplus_{k=0}^n \wedge^kV.
$$

%
The term `spin' is coming from the application of this representation and this group to theoretical physics.\\

Any dominant integral weight $\lambda$ can be written
$$
\lambda=\displaystyle\sum_{i=1}^n a_i\omega_i=\displaystyle\sum_{i=1}^n \lambda_i\varepsilon_i,
$$
where $a_i \in \mathbb N$ and $\lambda_i=a_i+\ldots+a_{n-1}+\frac{a_n}{2}$ if $i<n$ and $\lambda_n=\frac{a_n}{2}$.\\

If $a_n$ is even, the representation
$$
Sym ^{a_1}( V)\otimes Sym ^{a_2} (\wedge ^2 V)\otimes \dots \otimes Sym ^{a_{n-1}} (\wedge^{n-1}V)\otimes Sym ^{\frac{a_n}{2}}(\wedge ^{n}V)
$$
will contain an irreducible representation $\mathbb S^{[\lambda]}$.\\
If $a_n$ is odd, the tensor
$$
Sym ^{a_1}(V)\otimes Sym ^{a_2}(\wedge ^2 V)\otimes \dots \otimes Sym ^{a_{n-1}}(\wedge^{n-1}V)\otimes Sym ^{\frac{a_{n}-1}{2}}(\wedge^nV)\otimes V_{sp}
$$
will contain a copy of $\mathbb S^{[\lambda]}$.\\

Let us now give another way to build the simple submodule in $\otimes V$, using the Schur-Weyl duality.\\

For any choice of indices $i$ and $j$, satisfying $1\leq i<j\leq k$, define the contraction
\begin{eqnarray*}
\Phi_{ij}: \otimes^k V  &\longrightarrow & \otimes^{k-2} V\\
v_1\otimes\ldots\otimes v_k &\longmapsto& Q(v_i,v_j)v_1\otimes \ldots  \hat{v}_i \ldots \hat{v}_j \dots\otimes v_k.
\end{eqnarray*}
Let $V^{[0]}=\mathbb C$, $V^{[1]}=V$ and define
$$
V^{[k]}=\bigcap_{ij}\ker \big(\Phi_{ij}: \otimes^k V  \longrightarrow \otimes^{k-2} V\big)
$$
for $k\geq 2$.

For any partition $\lambda =(\lambda_1\geq \lambda_2 \dots \geq\lambda_{2n+1}\geq 0)$ of $k$, define the $\mathfrak{so}(2n+1)$-module $\mathbb{S}^{[\lambda]}$ of $\otimes^k V$, by
$$
\mathbb{S}^{[\lambda]}=V^{[k]}\cap \mathbb{S}^{\lambda},
$$
where $\mathbb{S}^{\lambda}$ is the $\mathfrak{sl}(2n+1)$-irreducible module with highest weight $\lambda$.\\

\begin{thm}\cite{FH}

\

For any $d\in\N$ there is an isomorphism of $(\mathfrak{so}(2n+1),S_k)$-modules 
$$
V^{[d]}=\bigoplus_{|\lambda|=d}\mathbb{S}^{[\lambda]}\otimes B_{\lambda}.
$$
For every partition $\lambda=(\lambda_1\geq \lambda_2 \dots \geq\lambda_{n}\geq 0)$ the $\mathfrak{so}(2n+1)$-module $\mathbb{S}^{[\lambda]}$ is the irreducible module with highest weight $\lambda=\lambda_1\varepsilon_1+\ldots+\lambda_n\varepsilon_n$.\\
\end{thm}

\subsection{Orthogonal semistandard columns}

\

The definition of semistandard columns for $\mathfrak{so}(2n+1)$ given in this section is equivalent but not identic to the definition given in \cite{L} by Cedric Lecouvey.

With the ordering $1<2<\dots<n<0<\overline{n}<\dots<\overline{1}$, a column is said to be semistandard if it satisfies the following properties:

\begin{itemize}
\item[1-] The entries are  increasing from top to bottom and if $t$ is not $0$ it appears at most one time,\\

\item[2-] Let $\mathcal C$ be such a column. We denote it ${\mathcal C}=\begin{array}{c}A\\ O\\ \overline{D}\end{array}$. In $O$ all entrees are $0$ and there is $A,~D$ are subsets of $[1,n]$,\\

\item[3-] Let $I=A\cap D$, then $I$ is in $\Delta_{A\cup D}$. We put $J=\delta_{A\cup D}(I)$,\\

\item[4-] $\#(A\cup D\cup J)+\#O\leq n$.\\
\end{itemize}

As for $\mathfrak{sp}(2n)$, we put $B=(A\setminus I)\cup J$, $C=(D\setminus I)\cup J$. Let $k=\#O$, there exists subsets in $[1,n]\setminus (A\cup D\cup J)$ having $k$ elements. We denote $K$ the greatest of these subsets.\\

We denote such a semistandard column by:
$$
{\mathcal C}=\begin{array}{c}A\\ O\\ \overline{D}\end{array}=f(A,O,D)=g(B,O,C).
$$

In addition to the admissible columns we have the spin columns, we denote them:
$$
\mathfrak C={\begin{array}{c} A\\ \overline{D}\end{array}}_{sp}=\mathfrak f(A,D),
$$
where $\#A+\#D=n$, $A\cap D=\emptyset$ and the entries increase strictly.\\

We will say that a column is admissible if it is semistandard and not spin and it is spin if it is semistandard and spin.\\

As in the $\mathfrak{sp}(2n)$ case, we define the double of a semistandard column. By definition, it is the two columns tableau:
$$
dble(\mathcal C)=dble\left(\begin{array}{c}
A\\
O\\
\overline{D}\\
\end{array}\right)=\begin{array}{cc}
A&B\\
K&\overline{K}\\
\overline{C}&\overline{D}
\end{array},\quad dble(\mathfrak C)=dble\left({\begin{array}{c}
A\\
\overline{D}\\
\end{array}}_{sp}\right)={\begin{array}{cc}
\begin{array}{c}
1\\
\vdots\\
n
\end{array}&
\begin{array}{c}
A\\
~\\
\overline{D}
\end{array}
\end{array}}_{sp},
$$
where it is understood that $A\cup K$, and $D\cup K$ are reordered to be written in a strictly increasing way.\\

\subsection{Relation with the Lecouvey's admisssible columns}

\

Let us mention that, for the non-spin case, this definition is not the Lecouvey's one. We recall that the admissible, non-spin columns in the sense of  Lecouvey are those such that:
\begin{itemize}
\item[1-]  The entries are  increasing from top to bottom and if $t$ is not $0$ it appears at most one time,\\

\item[2-] Let $\mathcal C_L$ such a column. We denote it ${\mathcal C_L}=\begin{array}{c}B\\ O\\ \overline{C}\end{array}$. In $O$ all entrees are $0$ and there is no zero in $B$ and $\overline{C}$,\\

\item[3-] Let $k=\#O$, and $J^1$ be the set $B\cap C\cup\{n+1,\dots,n+k\}$. We have $J^1\in\Gamma^{n+k}_{B\cup C}$.\\
\end{itemize}

Then we put $I^1=\gamma^{n+k}(J^1)$, $A^1=(B\setminus J^1)\cup I^1$, $D^1=(C\setminus J^1)\cup I^1$ and define the split of the column $\mathcal C_L$ as:
$$
split(\mathcal C_L)=\begin{array}{cc}
A^1&B\\
\overline{C}&\overline{D^1}
\end{array}.
$$

To prove the equivalence between the two notions, we define the subsets $I$ and $K$ in $[1,n]$, by:
$$
I=\gamma_{[1,n]\setminus(I^1\cup J)}(J),\quad K=I^1\setminus I.
$$

Let us remark that if $I=\{x_1<\dots<x_s\}$ and $K=\{z_1<\dots<z_k\}$, we do not have $x_i=x^1_i$ et $z_j=x^1_{s+j}$.\\

For instance, in $\mathfrak{so}(7)$, the following column is admissible in the Lecouvey sense:
$$
{\mathcal C_L}=
\begin{array}{|c|}
\hline
3\\
\hline
0\\
\hline
\overline{3}\\
\hline
\end{array}.
$$
Indeed, we have $n=3$, $k=1$, $B=C=\{3\}$, $J=\{3\}$, $J^1=\{3<4\}$, $I^1=\{1<2\}=\{x^1_1<x^1_2\}$, and $I^1\cup J^1=\{1,2,3,4\}$. Then $I=\{2\}=\{x_1\}$ and $K=\{1\}=\{z_1\}$,
$$
z_1=1\neq x_2^1=2,~~x_1=2\neq x^1_1=1.
$$

Put now $A=A^1\setminus K$, $D=D^1\setminus K$, we have $A=(B\setminus J)\cup I$, $D=(C\setminus J)\cup I$, and:
$$
split(\mathcal C_L)=\begin{array}{cc}
A&B\\
K&\overline{K}\\
\overline{C}&\overline{D}
\end{array}.
$$

It is moreover clear that the column $\mathcal C=\begin{array}{c}
A\\
O\\
\overline{D}
\end{array}$ is semistandard in the sense of the preceding section. In fact its double is the split of $\mathcal C_L$:
$$
dble(\mathcal C)=split(\mathcal C_L).
$$
Indeed, we have:

\begin{lem}

\

With our notations, we have:
$$
I=\gamma_{B\Delta C}(J).
$$
\end{lem}

\begin{proof}
Put $I'=\gamma_{B\Delta C}(J)=\{x'_1<\dots<x'_s\}$. By definition:
\begin{itemize}
\item[] $x'_s=\sup\{t\notin (B\Delta C\cup J)=B\cup C,~t<y_s\}$ and $x_s\notin B\cup C$ satisfies $x_s<y_s$ thus $x_s\leq x'_s$. If the $x_s<x'_s$ would hold, Lemma \ref{leftpart} applied to $I^1\cup J$ and $J$ would give $y_s<x'_s$ which is impossible, thus $x'_s=x_s$.\\

\item[] Suppose now $x'_s=x_s,\dots,x'_{i+1}=x_{i+1}$, then $x'_i=\sup\{t\notin B\cup C,~t<y_i,~t<x'_{i+1}\}$ and $x_i\notin B\cup C$ satisfies $x_i<y_i$ and $x_i<x_{i+1}=x'_{i+1}$, then $x_i\leq x'_i$. With the same argument as above, the only possibility is $x'_i=x_i$.\\
\end{itemize}

This proves $I'=I$.\\
\end{proof}

The preceding construction defines a map $\Phi$ from the set of admissible, non-spin, column in the sense of Lecouvey to the set of admissible column in our sense.

Conversely, if $\mathcal C=f(A,O,D)=g(B,O,C)$ is semistandard in our sense, we verify that the column $\mathcal C_L=\Psi(\mathcal C)=\begin{array}{c}
B\\
O\\
\overline{C}
\end{array}$ is admissible, in the sense of Lecouvey and non-spin.\\

By construction the mappings $\Phi$ and $\Psi$ are inverse one each other.\\ 

\begin{prop}

\

A basis for the fundamental module $\mathbb S^{[\omega_r]}$ is given by the non-spin semistandard columns with height $r$ if $r<n$, and the spin columns for $r=n$. 

The admissible, non-spin column, with height $n$ form a basis for the simple module $\mathbb S^{[2\omega_n]}$.\\
\end{prop}

We deduce as Lecouvey, that a column $\mathcal C$ (resp. $\mathfrak C$) is semistandard for $\mathfrak{so}(2n+1)$ if and only if $spl(\Psi(\mathcal C))$ (resp. $\Psi(\mathfrak C)=\mathfrak C$) is semistandard for $\mathfrak{sp}(2n)$, if and only if $dble(\mathcal C)$ is semistandard for $\mathfrak{sp}(2n)$.\\

\subsection{Orthogonal semistandard tableaux and shape algebra}

\

A tableau $T$ for $\mathfrak{so}(2n+1)$ is a succession of columns with decreasing heights such that, there are at most one spin column and in that case, it is the first starting from the left.

The double of this tableau is the tableau of $\mathfrak{sp}(2n)$ obtained by duplicate each column of $T$, arranged in their order.
$$\aligned
T&=\mathcal C_1\mathcal C_2\dots\mathcal C_r~\Longrightarrow~dble(T)=dble(\mathcal C_1)dble(\mathcal C_2)\dots dble(\mathcal C_r),\\
\text{resp.}&\\
T&=\mathfrak{C_1}\mathcal C_2\dots\mathcal C_r~\Longrightarrow~dble(T)=dble(\mathfrak{C_1})dble(\mathcal C_2)\dots dble(\mathcal C_r).
\endaligned
$$
We extend naturally $\Psi$ to any tableau and get:
$$\aligned
\Psi(T)&=\Psi(\mathcal C_1)\Psi(\mathcal C_2)\dots\Psi(\mathcal C_r)~\Longrightarrow~split(\Psi(T))=dble(T),\\
\text{resp.}&\\
\Psi(\mathfrak{C}T)&=\Psi(\mathfrak{C})\Psi(\mathcal C_1)\dots\Psi(\mathcal C_r)~\Longrightarrow~dble(\mathfrak C)split(\Psi(T))=dble(\mathfrak CT).
\endaligned
$$

We deduce the definition of a semistandard tableau for $\mathfrak{so}(2n+1)$:

\begin{defn}

\

A tableau $T$ is semistandard for $\mathfrak{so}(2n+1)$ if and only if its double $dble(T)$ is semistandard for $\mathfrak{sp}(2n)$.\\
\end{defn}

Since $dble(\mathfrak C)=\mathfrak C_0\mathfrak C$, where $\mathfrak C_0$ is the trivial column $\begin{array}{c}
1\\
\vdots\\
n	
\end{array}$, a tableau $T$ is semistandard if and only if $\Psi(T)$ is semistandard in the meaning of Lecouvey.\\

A dominant weight $\lambda$ corresponds now to a shape of tableaux, and the set $SS^{[\lambda]}$ of orthogonal semistandard tableaux with shape $\lambda$ is a basis for the simple module $\mathbb S^{[\lambda]}$. Similarly, the set $SS^{[\bullet]}$ of all orthogonal semistandard tableaux is a basis for the shape algebra $\oplus_\lambda\mathbb S^{[\lambda]}$ for the Lie algebra $\mathfrak{so}(2n+1)$.

\begin{rem}

In fact Kostant associates a notion of shape algebra for any reductive group $G$. In the algebraic case (see \cite{GW}) this algebra is explicitely realized as the space of affine regular functions on the quotient $\overline{N}\backslash G$, where $\overline{N}$ is the analytic subgroup whose Lie algebra is the opposite of $\mathfrak n$.

If $G$ is connected and simply connected, then then this notion of shape algebra is the geometric form of the shape algebra for $\mathfrak g=Lie(G)$. Thus here the shape algebra for $\mathfrak{so}(2n+1)$ is the geometric shape algebra for the group $Spin(2n+1,\mathbb  C)$.

If we restrict ourselves to the shape algebra for $SO(2n+1,\mathbb C)$, which has the same Lie algebra, we should obtain an algebra whose basis is given by the collection of all orthogonal semistandard tableaux without any spin column.\\ 

\end{rem}


\section{Orthogonal quasistandard tableaux}


\

Let us recall our definitions and notations. We say that a tableau $T$ is orthogonal semistandard ($T\in SS^{[\bullet]}$) if and only if $dble(T)$ is symplectic semistandard ($dble(T)\in SS^{\langle\bullet\rangle}$).\\

Now, it is clear, due to the structure of $splt(dble(T))$, that the condition $dble(T)\in SS^{\langle\bullet\rangle}$ is in fact equivalent to $dble(T)\in SS$.\\

\begin{defn}

\

Let $T$ be an orthogonal semistandard tableau, with shape $\lambda$ and $s\leq n$. Say that $T$ is not quasistandard in $s$ and write $T\in NQS_s^{[\lambda]}$, if and only if $dble(T)$ is not quasistandard in $s$, $dble(T)\in NQS_s^{\langle2\lambda\rangle}$.\\

Say $T$ is not quasistandard if and only if it exists $s$ such that $T$ is not quasistandard in $s$.\\

If it is not the case, we say that $T$ is quasistandard, and denote $T\in QS^{[\lambda]}$. We note $QS^{[\bullet]}$ the union of all the sets $QS^{[\lambda]}$.\\
\end{defn}

The definition $T\in NQS_s^{[\bullet]}$ is equivalent to the following condition, denoted $Hs$ (hypothesis in$s$): if $dble(T)=(dt_{ij})$,
\begin{itemize}
\item[1-] $dt_{s1}=s$, and there exists a column with height $s$ in $T$,
\item[2-] For all $j$ for which these quantities exist, $dt_{(s+1)j}>dt_{s(j+1)}$,
\end{itemize}

As in the $\mathfrak{sl}(n)$ and the symplectic case, we shall build a bijective map $p=push$ from $SS^{[\lambda]}$ to $\bigsqcup_{\mu\leq\lambda} QS^{[\mu]}$.\\

Since $T$ is orthogonal semistandard if and only if $dble(T)$ is symplectic semistandard, we shall use the `pouss' function defined in \cite{AK} for the symplectic case. This function is defined as a `maximal' use of the symplectic jeu de taquin $sjdt$.\\

But to `push' one step to the left a row $s$ in the tableau $T$, we have to push two steps to the left the row $s$ in the tableau $dble(T)$. So we need to verify, that after the first use of the symplectic jeu de taquin on $doubl(T)$, the result is still symplectic non quasistandard in $s$.\\  

On the other hand, the orthogonal jeu de taquin ($ojdt$) we shall study in the next sections is defined as a double use of the symplectic jeu de taquin. Therefore it is much more natural to directly use it to define the orthogonal `push' function.\\


\section{Direct expression for the orthogonal jeu de taquin $ojdt$}


\

\subsection{Definition of the $ojdt$}

\

Since the double of our orthogonal semistandard tableaux coincide with the split of the corresponding orthogonal semistandard tableaux defined by Lecouvey and since the orthogonal jeu de taquin is defined with the only use of the split form, we keep the Lecouvey definition for our setting.\\

\begin{defn}

\

Let $T$ be an orthogonal semistandard tableau, with shape $\lambda$, we suppose there is inside $T$, in the left and top corner, an empty Young diagram $S$, with shape $\mu$. To apply the $ojdt$ to $T\setminus S$, put a $\star$ in an interior corner of $S$, write down the double of these tableaux, getting a skew symplectic semistandard tableau $dble(T\setminus S)$ and two pointed boxes. Apply the symplectic jeu de taquin $sjdt$ successively for the two $\star$, the result is a symplectic semistandard tableau, which is the double of an orthogonal semistandard tableau $T'\setminus S'$. Put:
$$
ojdt(T\setminus S)=T'\setminus S'.
$$ 
\end{defn}

Indeed, Lecouvey proved in \cite{L} that the double action of the symplectic jeu de taquin on the double of $T\setminus S$ is the double of an orthogonal tableau.\\

\begin{rem}
The elementary move in the usual jeu de taquin is only a permutation of two succesive boxes inside $T\setminus S$, either horizontally (from left to right) or vertically (from the top to the bottom).

The elementary move in the symplectic jeu de taquin is very similar, except that, in the case of an horizontal move, we have to modify the two concerned columns.

The elementary move in the orthogonal jeu de taquin can be a permutation along a diagonal, followed by a modification of the columns, as the following example shows:\\

$\begin{array}{l}
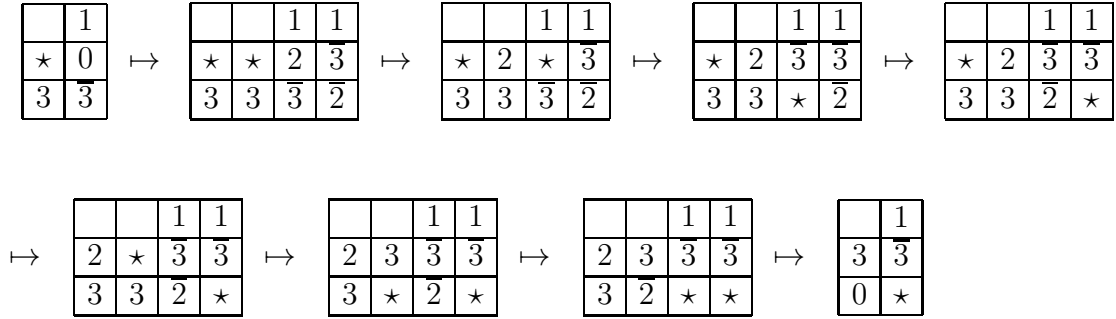
\begin{array}{|c|c|}
\hline
 &1\\
\hline
 \star&0\\
\hline
3&\overline 3\\
\hline
\end{array}
\end{array}\mapsto~~\begin{array}{l}
\begin{array}{|c|c|c|c|}
\hline
 &&1&1\\
\hline
 \star&\star&2&\overline 3\\
\hline
3&  3&\overline3&\overline 2\\
\hline
\end{array}
\end{array}\mapsto~~\begin{array}{l}
\begin{array}{|c|c|c|c|}
\hline
 &&1&1\\
\hline
 \star&2&\star&\overline 3\\
\hline
3&  3&\overline3&\overline 2\\
\hline
\end{array}
\end{array}\mapsto~~\begin{array}{l}
\begin{array}{|c|c|c|c|}
\hline
 &&1&1\\
\hline
 \star&2&\overline3&\overline 3\\
\hline
3&  3&\star&\overline 2\\
\hline
\end{array}
\end{array}\mapsto~~\begin{array}{l}
\begin{array}{|c|c|c|c|}
\hline
 &&1&1\\
\hline
 \star&2&\overline3&\overline 3\\
\hline
3&  3&\overline 2&\star\\
\hline
\end{array}
\end{array}$\\\\\\

$\mapsto~~\begin{array}{l}
\begin{array}{|c|c|c|c|}
\hline
 &&1&1\\
\hline
 2&\star&\overline3&\overline 3\\
\hline
3&  3&\overline 2&\star\\
\hline
\end{array}
\end{array}
\mapsto~~\begin{array}{l}
\begin{array}{|c|c|c|c|}
\hline
 &&1&1\\
\hline
 2&3&\overline3&\overline 3\\
\hline
3&  \star&\overline 2&\star\\
\hline
\end{array}
\end{array}\mapsto~~\begin{array}{l}
\begin{array}{|c|c|c|c|}
\hline
 &&1&1\\
\hline
 2&3&\overline3&\overline 3\\
\hline
3&  \overline 2&\star&\star\\
\hline
\end{array}
\end{array}\mapsto~~
\begin{array}{l}
\begin{array}{|c|c|}
\hline
 &1\\
\hline
 3&\overline 3\\
\hline
0&\star\\
\hline
\end{array}
\end{array}$\\
\end{rem}

From now on, we consider only the `horizontal situation' ($HS$ hypothesis):
\begin{itemize}
\item[1-] In the tableau $dble(T\setminus S)$, the double star are in the row $s$,
\item[2-] $dble(T\setminus S)=(dt_{ij})$ has two columns with height $s$,
\item[3-] $dt_{s(j+1)}<dt_{(s+1)j}$ for each $j$ where these two entries exist.
\end{itemize}

In this situation we can describe the elementary move.

\begin{thm}

\

Suppose the skew tableau $T\setminus S=(t_{ij})$ and the star are in the $HS$ situation, then
\begin{itemize}
\item[1-] The situation $t_{sj}=0$ and $t_{(s+1)j}=0$ is impossible, for any $j$.
\item[2-] The move is always horizontal,
\item[3-] For each elementary move on $T\setminus S$, with our notation, the move is exactly like for the horizontal move in the $sjdt$, with the addition that if $t_{s(j+1)}=0$ and $\star$ is in the $(s,j)$ box, then the move is simply a permutation of these two entries.\\
\end{itemize}
\end{thm}

The assertions of the theorem mean there is only horizontal elementary moves, each of them being:
\begin{itemize}
\item[Move 1-] if $(t_{sj},t_{s(j+1)})=(\star,a)$ with $a$ unbarred, the move is:
$$
g(B_j\cup\{\star\},O_j,C_j)f(A_{j+1}\cup\{a\},O_{j+1},D_{j+1})\mapsto g(B_j\cup\{a\},O_j,C_j)f(A_{j+1}\cup\{\star\},O_{j+1},D_{j+1}),
$$

\item[Move 2-] if $(t_{sj},t_{s(j+1)})=(\star,\overline{c})$ with $\overline{c}$, the move is:
$$
f(A_j,O_j,D_j\cup\{\star\})g(B_{j+1},O_{j+1},C_{j+1}\cup\{c\})\mapsto f(A_j,O_j,D_j\cup\{c\})g(B_{j+1},O_{j+1},C_{j+1}\cup\{\star\}),
$$

\item[Move 3-] if $(t_{sj},t_{s(j+1)})=(\star,0)$, the move is:
$$
f(A_j,O_j\cup\{\star\},D_j)g(B_{j+1},O_{j+1}\cup\{0\},C_{j+1})\mapsto f(A_j,O_j\cup\{0\},D_j)g(B_{j+1},O_{j+1}\cup\{\star\},C_{j+1}),
$$
\end{itemize}

\subsection{Proof of the theorem}

\

Let $\mathcal C=f(A,O,D)=g(B,O,C)$ be a column in $T\setminus S$, we note its double:
$$
dble(\mathcal C)=\begin{array}{cc}
A&B\\
K&\overline{K}\\
\overline{C}&\overline{D}
\end{array}=\begin{array}{cc}
E&B\\
\overline{C}&\overline{F}
\end{array},
$$
and recall the definition of $I=A\cup D=\{x_1<\dots<x_r\}$, $J=B\cap C=\{y_1<\dots<y_r\}$, $K=\{z_1<\dots<z_k\}$. Suppose the $s$ row in this double is: $\boxed{u_s}\boxed{v_s}$. In this section, like in the $HS$ situation, we assume: 
$$
v_s<u_{s+1}.
$$

The proof of the theorem needs the following technical propositions.\\

\begin{prop}\label{prop0}

\

\begin{itemize}
\item[1.] If $u_s=e_s$ and $v_s=b_s$ are unbarred, then $u_s$ is in $A$, and: either $e_s=b_s$ are in $A\setminus I$ or $e_s\in I$, $e_s=x_i$ and $b_s=y_i$.\\
\item[2.] If $u_s=e_s$ is unbarred and $v_s=\overline{f_s}$ is barred, then $e_s$ and $f_s$ are in $K$, $f_s=z_1$ and $e_s=z_k$.\\
\item[3.] If $u_s=\overline{c_s}$ and $v_s=\overline{f_s}$ are barred, then $f_s\in D$ and, either $c_s=f_s\in D\setminus I$ or $f_s\in I$, $f_s=x_i$ and $c_s=y_i$.\\
\end{itemize}
\end{prop}

\begin{proof}

\

\noindent{\bf Case 1} If $e_s=b_s$, then $b_s\notin K$, $e_s\notin K$, $e_s\in A$, and $e_s\in A\setminus B=A\setminus I$.

If $e_s<b_s$, let us assume $e_s\in K$, since $K$ is the largest possible subset in $[1,n]\setminus(A\cup C)$, this implies $[e_s,n]\subset A\cup C\cup K=E\cup C$ but $b_s\notin E$ since $e_s<b_s<e_{s+1}$, if $b_s$ was in $C$, then $b_s\in C\setminus D=C\setminus J=D\setminus I$, thus $b_s\in D$, which is impossible, therefore $u_s=e_s$ is in $A$.

Now $b_s\in B\setminus A=J$, there exists $i$ such that $b_s=y_i$, let us consider $x_i$. We put $I^{\leq w}=I\cap[\!1,w\!]$. Then Lemme 5.1 in \cite{AK} implies 
$x_i\in I^{\leq b_s}=I^{\leq e_s}$, then $x_i\leq e_s$. Now, if $x_i<e_s$, then $x_i\in I^{\leq e_{s-1}}$, but $b_s$ is not in $J^{\leq e_{s-1}}$, and this is a contradiction with Lemme 5.1 in \cite{AK}.\\

\noindent{\bf Case 2} By the assumption $v_s<u_{s+1}$, $e_s=\sup(A\cup K)$ and $f_s>\sup(C)$. Therefore $f_s$ is not in $C$, thus $f_s\notin B$, suppose $f_s\in D$, then $f_s\in D\setminus C=D\setminus I$, this means $f_s\in A$, $f_s\leq e_s$, now if $f_s$ is in $K$, $f_s\leq e_s$ too.

Let $w\in D$, if $w\in D\setminus I$, then $w\in C$, $w<f_s$, if $w\in I$, then $w=x_i<w'=y_i$ in $J\subset C$, thus $w<w'<f_s$. The relation $w<f_s$ holds in any case. But there is as much entries strictly below $f_s$ and in $C$, thus $D=F\cap[1,f_s-1]$ and $K=\{z_1<\dots<z_k\}=F\cap[f_s,n]$. Especially, $f_s=z_1$.

On the other hand, if $e_s$ is in $A$, $e_s>z_k$, therefore $e_s\in A\setminus C=I$, $e_s\in D$, and $e_s<z_k$, which is impossible, then $e_s\in K$, $e_s=z_k$.\\

\noindent{\bf Case 3} This case is the symmetric of case 1. The proof is the same {\sl mutatis mutandis}.\\
\end{proof}

Let us now define the element $v'_s$ as follows:
\begin{itemize}
\item[1-] If $v_s=b_s$ is in $B$, we put: $v'_s=v_s$ if $v_s\notin C$ and $v'_s=\gamma_{E\cup(C\setminus\{v_s\})}(v_s)$ if $v_s\in C$.\\

\item[2.] If $v_s=\overline{f_s}$ with $f_s\in F$, and $u_s=e_s\in E$, we put $v'_s=\delta_{(E\setminus\{f_s\})\cup C}(f_s)$.\\

\item[3.] If $v_s=\overline{f_s}$ with $f_s\in F$, and $u_s=\overline{c_s}$, with $c_s\in C$, we put $v'_s=\overline{f_s}$ if $f_s\notin E$
and $v'_s=\overline{\delta_{(E\setminus\{f_s\})\cup C}(f_s)}$ if $f_s\in E$.\\
\end{itemize}

\begin{cor}

\

We have $u_s=v'_s$, or the $s$ row in $dble(\mathcal C)$ is $\boxed{v'_s|v_s}$.\\
\end{cor}

\begin{proof}

In the case 1, we saw that either $e_s=b_s=v'_s$ or $e_s$ is the greatest element in $I^{\leq b_s}$, that means $u_s=e_s=\sup\{t\notin(E\cup C)^{\leq b_s},~t<b_s\}=v'_s$.\\

In the case 2, we saw that $e_s\leq f_s$, and $e_s=z_k$, $f_s=z_1$. Then $[z_1,n]\subset E\cup C$, and $\delta_{(E\setminus\{e_s\})\cup C}(f_s)=e_s$.\\

The case 3 is similar to the case 1.\\
\end{proof}

Let us  now prove the theorem. We suppose that the orthogonal jeu de taquin was well defined and  was always moving horizontally until some point, where the star is in the row $s$ and some column.\\

\begin{lem}

\

In the tableau $T$, it is impossible to have one of the following disposition, for any $j$:
$$
T=\begin{array}{lcc}
~~~~~~~\dots~~~~~~~~~~~~~~~&\vdots&~~~~~~~~~\dots~~~~~~~~~~~~\\
_s&0&\\
\hline
&0&\\
&\vdots\\
&(j)&\\
\end{array}\quad\text{or}\quad
T=\begin{array}{lccc}
~~~~~~~\dots~~~~~~~~~~~~~~~&\vdots&\vdots&~~~~~~~~~\dots~~~~~~~~~~~~\\
_s&&0&\\
\hline
&a&&\\
&\vdots&\vdots\\
&(j-1)&(j)&\\
\end{array}.
$$
with $a$ unbarred.
\end{lem}

\begin{proof}
We assume that $T$ has the above first form then the tableau $dble(T)$ has the following disposition in its columns $2j$ and $2j+1$,
$$
dble(T)=\begin{array}{lccc}
~~~~~~~\dots~~~~~~~~~~~~~~~&\vdots&\vdots&~~~~~~~~~\dots~~~~~~~~~~~~\\
_s&a'_s&\overline{d'_s}&\\
\hline
&a'_{s+1}&\overline{d'_{s+1}}&\\
&\vdots&\vdots&\\
&(2j)&(2j+1)&\\
\end{array}.
$$
Since $\overline{d'_s}>a'_{s+1}$, this is in contradiction with our assumption $HS$.\\

Similarly, suppose  the tableau has the second form, then $dble(T)$ is
$$
dble(T)=\begin{array}{lccccc}
~~~~~~~\dots~~~~~~~~~~~~~~~&\vdots&\vdots&\vdots&\vdots&~~~~~~~~~\dots~~~~~~~~~~~~\\
_s&&&e_s&\overline{f_s}&\\
\hline
&e'_{s+1}&b_{s+1}&&&\\
&\vdots&\vdots&\vdots&\vdots&\\
&(2j-2)&(2j-1)&(2j)&(2j+1)&\\
\end{array}.
$$

\end{proof}

The proposition said that in the columns $(2j)(2j+1)$, $K_j=\{z_1<\dots<z_r\}$, $f_s=z_1$, $e_s=z_r$. But, in these columns, $f>\sup(C_j)$, this implies $[f,n]\subset K_j\cup C_j\cup A_j$, thus $n\in K_j\cup C_j\cup A_j$, and $n\geq f>\sup C_j$, therefore, $n\in K_j\cup A_j$. On the other hand, we have $e_s<b_{s+1}\leq n$, thus $e_s=\sup(A_j\cup K_j)<n$. This is impossible.\\

\subsubsection{Study of the case $\boxed{\star|0}$}

\

We  now assume that there exists, on the $s$ row, a star and a zero to the right of the star. Then doubling the two concerned columns is the following:
$$
\begin{array}{lccc}
\hskip 0.5cm&A_1&\begin{array}{c}
A_2\\ O_2
\end{array}&\hskip 0.5cm\\
_s&\star&0&\\
\hline
&O_1&\overline{D_2}&\\
&\overline{D_1}&&\\
\end{array}~~\mapsto~~
\begin{array}{lccccc}
\hskip 0.5cm&E_{11}&B_1&E_2& \begin{array}{c}
B_2\\
\overline{F_{21}}
\end{array}&\hskip 0.5cm\\
_s&\star&\star&e&\overline{f}&\\
\hline
&\begin{array}{c}
E_{12}\\
\overline{C_1}
\end{array}&\overline{F_1}&\overline{C_2}&\overline{F_{22}}&	
\end{array}
$$

We therefore have $e=\sup(E_2\cup\{e\})$, and since the tableau is symplectic semistandard, $\sup(E_2)\geq\sup(B_1)$, $e>\sup(B_1)$.\\

\begin{lem}\label{xis central}

\

If $e\in F_1$, we set $x=\gamma_{B_1\cup F_1}(e)$, then $\inf(E_{12})>x>\sup(B_1)$.\\
\end{lem}

\begin{proof}

Either we are at the starting point, there is no $B_1$, no $A_1$ and $E_{12}=K_1$, $F_1=K_1\cup C_1$. By the definition of $K_1$, since $x$ is not in $F_1$, then $x<\inf(K_1)$, this proves the lemma in this case.\\

Or there was a preceding step, where we had:
$$
\begin{array}{lcccccc}
\hskip 0.5cm&&\tilde{E_{11}}&B_1&E_2& \begin{array}{c}
B_2\\
\overline{F_{21}}
\end{array}&\hskip 0.5cm\\
_s&\star&\tilde{e}&\tilde{b}&e&\overline{f}&\\
\hline
&&\begin{array}{c}
\tilde{E_{12}}\\
\overline{\tilde{C_1}}
\end{array}&\overline{F_1}&\overline{C_2}&\overline{F_{22}}&	
\end{array}
$$
Our induction hypothesis says that the move is horizontal, the element $\tilde{e}$ leaves its column. Since, in the first column, all the entries are distinct, in the $sjdt$, this column does not change. Then in the next symplectic step in $sjdt$, the star moves still horizontally, the element $\tilde{b}$ leaves the second column, which as above does not change. $\tilde{b}$ goes inside the first column, under the name $y=\gamma_{B_1\cup F_1}(\tilde{b})$, since $\tilde{b}<\tilde{e_{12}}=\inf(\tilde{E_{12}})$, the first column becomes:
$$
g(\tilde{E_{11}}\cup\tilde{E_{12}}\cup\{\tilde{b}\},\tilde{C_1})=f(\tilde{E_{11}}\cup\tilde{E_{12}}\cup\{y\},(\tilde{C_1}\setminus\{\tilde{b}\})\cup\{y\}).
$$

Suppose $\tilde{b}=\tilde{e}$, then $y=\tilde{b}$, $\tilde{E_{11}}=E_{11}$, $\tilde{E_{12}}=E_{12}$, $\tilde{C_1}=C_1$.

Suppose now $\tilde{b}>\tilde{e}$, then $\tilde{e_{12}}>\tilde{b}\geq y\geq\tilde{e}>\tilde{e_{11}}$, then $\tilde{E_{11}}=E_{11}$, $\tilde{E_{12}}=E_{12}$, $\tilde{C_1}=C_1\setminus\{y\}\cup\{\tilde{b}\}$.

So in every case, we get:
$$
\begin{array}{lcccccc}
\hskip 0.5cm&&E_{11}&B_1&E_2& \begin{array}{c}
B_2\\
\overline{F_{21}}
\end{array}&\hskip 0.5cm\\
_s&\star&y&\star&e&\overline{f}&\\
\hline
&&\begin{array}{c}
E_{12}\\
\overline{C_1}
\end{array}&\overline{F_1}&\overline{C_2}&\overline{F_{22}}&	
\end{array}
$$

Now $b_1=\sup(B_1)<\tilde{b}\leq e$. Therefore:
\begin{itemize}
\item[] either $e=\tilde{b}\notin B_1\cup F_1$ thus $e_{12}>x=e=\tilde{b}>b_1$, 
\item[] or $e>\tilde{b}\notin B_1\cup F_1$ and by the definition of $x$, $x\geq\tilde{b}>b_1=\sup(B_1)$.
\end{itemize}
That means $x>\sup(A_1)$, $E_{11}=A_1$, $x<\inf(K_1)=e_{12}=\inf(E_{12})$.\\
\end{proof}

In the next step of the symplectic jeu de taquin $sjdt$, $e$ leaves the third column and becomes $x$. We put $F'_1=(F_1\setminus\{e\})\cup\{x\}$, and get, with the lemma, the following sequence of steps:
$$\aligned
\mapsto&\begin{array}{lccccc}
\hskip 0.5cm&E_{11}&B_1&E_2&\begin{array}{c}
B_2\\
\overline{F_{21}}
\end{array}&\hskip 0.5cm\\
_s&\star&x&\star&\overline{f}&\\
\hline
&\begin{array}{c}
E_{12}\\
\overline{C_1}
\end{array}&\overline{F'_1}&\overline{C_2}&\overline{F_{22}}&	
\end{array}\\
&\\
\mapsto&\begin{array}{lccccc}
\hskip 0.5cm&E_{11}&B_1&E_2&\begin{array}{c}
B_2\\
\overline{F_{21}}
\end{array}&\hskip 0.5cm\\
_s&\star&x&\overline{f}&\star&\\
\hline
&\begin{array}{c}
E_{12}\\
\overline{C_1}
\end{array}&\overline{F'_1}&\overline{C_2}&\overline{F_{22}}&	
\end{array}.
\endaligned
$$

Now we saw $\sup(E_{11})<x<\inf(E_{12})$. Since $x\notin F_1$, then $x\notin D_1$. Suppose that $x\in C_1\setminus D_1=J_1=B_1\cap C_1$, then $x$ is in $B_1$ and according to the Lemma this is wrong. So $x$ is not in $C_1$ and the next step is:
$$\aligned
\mapsto&\begin{array}{lccccc}
\hskip 0.5cm&E_{11}&B_1&E_2&\begin{array}{c}
B_2\\
\overline{F_{21}}
\end{array}&\hskip 0.5cm\\
_s&x&\star&\overline{f}&\star&\\
\hline
&\begin{array}{c}
E_{12}\\
\overline{C_1}
\end{array}&\overline{F'_1}&\overline{C_2}&\overline{F_{22}}&	
\end{array}
\endaligned
$$

Now, if $k_2=\#O_2$, $f$ is the $k_2^{\rm ~st}$ element of $F_2$ starting from the top. We know that $\#[1,n]\setminus(A_2\cup D_2\cup J_2)=\#[1,n]\setminus(B_2\cup D_2)\geq k_2$. Denote $X$ the greatest subset having $k_2$ elements in $F_2$ and $Y$ the greatest subset having $k_2$ elements in $[1,n]\setminus B_2$. We have $X=\{x_1<x_2<\dots<x_{k_2}\}$ and $Y=\{y_1<\dots<y_{k_2}\}$.

By construction, if $K_2=\{z_1<\dots<z_{k_2}\}$, then $z_i\leq y_i$.\\

Consider an element $y_i$. If $y_i\in D_2$, then $y_i\in F_2$. If $y_i\notin D_2$, then $y_i$ is in $[1,n]\setminus(D_2\cup B_2)$, so $y_i$ is a $z_j$, for some $j\leq i$, $y_i\in K_2$, and $y_i\in F_2$: this proves $Y\subset F_2$.

Since $X$ is the greatest of the subsets in $F_2$, $y_i\leq x_i$ but since $X\subset F_2\subset[1,n]\setminus B_2$ and $Y$ is the greatest such subset, $x_i\leq y_i$. 

That means $X=Y$, and especially $f=x_1=y_1$.\\

On the other hand, if $e\in A_2$, Since $e\notin C_2$, then $e\in A_2\setminus C_2=I_2=A_2\cap D_2$ so $e\in F_2=K_2\cup D_2$. But if $e\notin A_2$, then $e\in K_2$, so $e\in F_2$.

In all cases $e$ is the greatest element in $F_2$.

Indeed, if there exists $t$ in $F_2$ such that $t>e$, then either $t\in K_2$, then $t\in E_2\cup\{e\}$, which is impossible, or $t\in D_2$ and $t\notin C_2$, then $t\in D_2\setminus C_2=I_2\subset A_2\subset E_2\cup\{e\}$, which is still impossible.

So $e=x_{k_2}=y_{k_2}$.\\

Let $t>f$ and $t\neq e$. Then either $t\in B_2$ or $t\in Y\subset F_2$, in all cases
$$
t\in(F_2\cup B_2)\setminus\{e\}=\left((E_2\cup\{e\})\cup C_2\right)\setminus\{e\}=E_2\cup C_2.
$$
So $e=\delta_{E_2\cup C_2}(f)$.\\

Then, in the next step, we set $E'_2=(E_2\setminus\{f\})\cup\{e\}$, and get, as $\overline{f}$ leaves the column 3,
$$\aligned
\mapsto&\begin{array}{lccccc}
\hskip 0.5cm&E_{11}&B_1&E_2&\begin{array}{c}
B_2\\
\overline{F_{21}}
\end{array}&\hskip 0.5cm\\
_s&x&\star&\overline{f}&\star&\\
\hline
&\begin{array}{c}
E_{12}\\
\overline{C_1}
\end{array}&\overline{F'_1}&\overline{C_2}&\overline{F_{22}}&	
\end{array}\\
&\\
\mapsto&\begin{array}{lccccc}
\hskip 0.5cm&E_{11}&B_1&E'_2& \begin{array}{c}
B_2\\
\overline{F_{21}}
\end{array}&\hskip 0.5cm\\
_s&x&\overline{e}&\star&\star&\\
\hline
&\begin{array}{c}
E_{12}\\
\overline{C_1}
\end{array}&\overline{F'_1}&\overline{C_2}&\overline{F_{22}}&	
\end{array}\\
\endaligned
$$


Indeed, our assumption $HS$ gives: $f<\inf(C_2)\leq\inf(F_1)$, and $e>f$, thus $f<\inf(F'_1)$, on the other side, we saw that $e>\sup(B_1)$.\\

Now, the two first columns give a $\mathfrak{so}$ column with one 0 more. More precisely, the corresponding first $\mathfrak{so}$ column becomes:
$$
f(A_1,O_1\cup\{\star\},D_1)=
\begin{array}{c}
A_1\\
\star\\
O_1\\
\overline{D_1}\
\end{array}~\mapsto~g(B_1,O_1\cup\{0\},C_1)=\begin{array}{c}
A_1\\
0\\
O_1\\
\overline{D_1}\
\end{array}=f(A_1,O_1\cup\{0\},D_1).
$$

The two last columns is also the double of a $\mathfrak{so}$ column, in fact it gives:
$$
f(A_2,O_2,D_2)=
\begin{array}{c}
A_2\\
O'_2\\
0\\
\overline{D_2}\
\end{array}~\mapsto~g(B_2,O_2\setminus\{0\}\cup\{\star\},C_2)=\begin{array}{c}
A_2\\
O'_2\\
\star\\
\overline{D_2}\
\end{array}=f((A_2,O_2\setminus\{0\})\cup\{\star\},D_2)
$$

This is exactly the Move 3 case described after the theorem.\\

\subsubsection{Study of the case $\boxed{\star|a}$}
\

Let us now study the case:
$$
\begin{array}{lccc}
\hskip 0.5cm&A_{11}&A_{21}&\hskip 0.5cm\\
_s&\star&a&\\
\hline
&A_{12}&A_{22}&\\
&O_1&O_2\\
&\overline{D_1}&\overline{D_2}&
\end{array}~~\sim~~\begin{array}{lccccc}
\hskip 0.5cm&E_{11}&B_{11}&E_{21}&B_{21}&\hskip 0.5cm\\
_s&\star&\star&e&b&\\
\hline
&E_{12}&B_{12}&E_{22}&B_{22}\\
&\overline{C_1}&\overline{F_1}&\overline{C_2}&\overline{F_2}&
\end{array}.
$$

\noindent
\underbar{\bf Step 1} We move the first $\star$ for 2 steps, the second one for one step.\\

Then the entry $e$ enters into the column 2 which becomes $g(B_1\cup\{e\},F_1)$ or, if $x=\gamma_{B_{11}\cup B_{12}\cup F_1}(e)$, and $F'_1=(F_1\setminus\{e\})\cup\{x\}$,
$$
\mapsto~\begin{array}{lccccc}
\hskip 0.5cm&E_{11}&B_{11}&E_{21}&B_{21}&\hskip 0.5cm\\
_s&\star&x&\star&b&\\
\hline
&E_{12}&B_{12}&E_{22}&B_{22}\\
&\overline{C_1}&\overline{F'_1}&\overline{C_2}&\overline{F_2}&
\end{array}.
$$

The argument of Lemma \ref{xis central} tell us that $x>\sup(B_{11})=b_{11}$ and $x\leq e<b_{12}=\inf(B_{12})$. Then $x$ is not in $C_1$, as above, it enters in the first column:
$$
\mapsto~\begin{array}{lccccc}
\hskip 0.5cm&E_{11}&B_{11}&E_{21}&B_{21}&\hskip 0.5cm\\
_s&x&\star&\star&b&\\
\hline
&E_{12}&B_{12}&E_{22}&B_{22}\\
&\overline{C_1}&\overline{F'_1}&\overline{C_2}&\overline{F_2}&
\end{array}.
$$

Similarly, $b$ goes in the column 3, under the name $y=\gamma_{E_{21}\cup E_{22}\cup C_2}(b)$:
$$
\mapsto~\begin{array}{lccccc}
\hskip 0.5cm&E_{11}&B_{11}&E_{21}&B_{21}&\hskip 0.5cm\\
_s&x&\star&y&\star&\\
\hline
&E_{12}&B_{12}&E_{22}&B_{22}\\
&\overline{C_1}&\overline{F'_1}&\overline{C'_2}&\overline{F_2}&
\end{array}.
$$
As before, $y$ is in this row, $F'_1=(F_1\setminus\{e\})\cup\{x\}$.\\

\noindent
\underbar{\bf Step 2} Let us now prove $e$ is not in $K_2$:\\

Assume $e\in K_2$. Then $b\neq e$, with our hypothesis $HS$, $b$ is not in $E_2$, so $b\notin A_2$, and then $b\in J_2$. Set $I_{21}=\{t\in I_2, t<e\}$, $A_{21}=\{t\in A_2, t<e\}$, same definition for $K_{21}$, $B_{21}$. Then
$$
E_{21}\setminus I_{21}=(A_{21}\setminus I_{21})\cup K_{21},\qquad A_{21}\setminus I_{21}\subset B_{21},\qquad K_{21}\cap B_{21}=\emptyset.
$$
Since $b\in J_2$, there exists $t\in I_{21}$, $t=x_i$ such that $b=y_i$, so $\#(J_2\cap B_{21})<\#I_{21}$, or $\#(B_{21}\setminus J_2)>\#(A_{21}\setminus I_{21})$. Then there exists $a\in B_{21}\setminus J_2\subset A_2\setminus I_2$ such that $a\notin A_{21}$, $b>a>e$, so we have
$$
e_{22}=\inf(E_{22})\leq a<b,
$$
which is in contradiction with our hypothesis $HS$.\\

\noindent
\underbar{\bf Step 3} Let us show that $y=e$.\\

If $e=b$, then $b\in A_2\setminus I_2=B_2\setminus J_2$, $b\notin C_2$, so $y=b=e$.

If $e<b$, since $e\notin E_{21}\cup E_{22}\cup C_2$, then $e\in\{t\notin E_{21}\cup E_{22}\cup C_2,~t<b\}$, so
$e\leq y=\sup\{t\notin E_{21}\cup E_{22}\cup C_2,~t<b\}$.

Assume now $e<y$.

Since $b\in B_2$ and $b\notin E_2$, $b\in B_2\setminus A_2=J_2$. Let $J_2=\{y_1<\dots<y_i=b<\dots<y_r\}$ and $I_2=\{x_1<\dots<x_r\}$.

Put $A\triangle D=(A\cup D)\setminus(A\cap D)$, then:
$$
x_i=\sup\{t\notin A_2\triangle D_2,~t<x_{i+1},~t<y_i=b\},
$$

Let us prove that $y\geq x_{i+1}$. In fact if $y<x_{i+1}$, then $y\in \{t\notin E_{21}\cup E_{22}\cup C_2,~t<b\}$, so $y\leq x_i$ and
$$
e<y\leq x_i<b<e_{22}=\inf(E_{22}).
$$
But this is impossible because there is no elements in $E$ between  $e$ and $b_{22}$. So $y\geq x_{i+1}$. Since $x_{i+1}\in E_2$ et $y\notin E_2$, we have $y>x_{i+1}$.

Now $x_{i+1}=\sup\{t\notin A_2\triangle D_2,~t<x_{i+2},~t<y_{i+1}\}$.

Following the same argument we prove $y\geq x_{i+2}$:

As before: if $y<x_{i+2}$, then $e<y\leq x_{i+1}$, which is wrong, so $y\geq x_{i+2}$, $y>x_{i+2}$, and so on... Finally:
$$
y>x_r=\sup\{t\notin A_2\triangle D_2,~t<y_r\},
$$
But $y\notin A_2\triangle D_2$, $y<b\leq y_r$, so $y\leq x_r$, which is contradiction.

So the hypothesis $e<y$ is wrong: $e=y$.\\

\noindent
\underbar{\bf Step 4} End of the move.

Then we obtain $e\notin F'_1$ and :
$$
\begin{array}{lccccc}
\hskip 0.5cm&E_{11}&B_{11}&E_{21}&B_{21}&\hskip 0.5cm\\
_s&x&e&\star&\star&\\
\hline
&E_{12}&B_{12}&E_{22}&B_{22}\\
&\overline{C_1}&\overline{F'_1}&\overline{C'_2}&\overline{F_2}&
\end{array}~~\sim~~g(B_1\cup\{e\},O_1,C_1)f((A_2\setminus\{e\})\cup\{\star\},O_2,D_2).
$$

This is exactly the Move 1 case described after the theorem.\\

\subsubsection{Study of the case $\boxed{\star|\overline{d}}$}
\

Let us now study the last case:
$$
\begin{array}{cc}
A_1&A_2\\
O_1&O_2\\
\overline{D_{11}}&\overline{D_{21}}\\
\star&\overline{d}\\
\overline{D_{12}}&\overline{D_{22}}\\
\end{array}
$$

Doubling the tableau, we get:
$$
\sim \begin{array}{lccccc}
\hskip 0.5cm&E_1&B_1&E_2&B_2&\hskip-0.2cm\\
&\overline{C_{11}}&\overline{F_{11}}&\overline{C_{21}}&\overline{F_{21}}&\\
_s&\star&\star&\overline{c}&\overline{f}&\\
\hline
&\overline{C_{12}}&\overline{F_{12}}&\overline{C_{22}}&\overline{F_{22}}&\\
\end{array}
$$

\noindent
\underbar{\bf Step 1}  We move the first $\star$ for 2 steps, the second one for one step.

We get:
$$
\longmapsto~~\begin{array}{lccccc}
\hskip 0.5cm&E_1&B_1&E_2&B_2&\hskip-0.2cm\\
&E_1&\overline{F_{11}}&E_2&\overline{F_{21}}&\\
&\overline{C_{11}}&\overline{F_{11}}&\overline{C_{21}}&\overline{F_{21}}&\\
_s&\star&\overline{c}&\overline{f}&\star&\\
\hline
&\overline{C_{12}}&\overline{F_{12}}&\overline{C_{22}}&\overline{F_{22}}&\\
\end{array}.
$$
Now $\overline{c}$ enters the first column under the name $x=\delta_{B_1\cup F_{11}\cup F_{12}}(c)$. We put $B'_1=(B_1\setminus\{c\})\cup\{x\}$.\\

Remark that $x\notin C_{11}\cup C_{12}$ and $x\notin B_1\cup F_1=B_1\cup K_1\cup D_1$ by construction. Therefore $x$ is not in $E_1$. In fact,
if $x\in E_1=A_1\cup K_1$, then since $x\notin K_1$, $x\in A_1$, so
$$
x\in A_1\setminus D_1=A_1\setminus I_1=B_1\setminus J_1\subset B_1,
$$
which is impossible, so $x\notin E_1$ and we have $c_{11}=\inf(C_{11})>x>c_{12}=\sup(C_{12})$, then $x$ still remains on the same row:
$$
\longmapsto~~\begin{array}{lccccc}
\hskip 0.5cm&E_1&B'_1&E_2&B_2&\hskip-0.2cm\\
&\overline{C_{11}}&\overline{F_{11}}&\overline{C_{21}}&\overline{F_{21}}&\\
_s&\overline{x}&\star&\overline{f}&\star&\\
\hline
&\overline{C_{12}}&\overline{F_{12}}&\overline{C_{22}}&\overline{F_{22}}&\\
\end{array}.
$$

Let us be more precise: since $\sup(C_{22})<f<\inf(C_{21})$, we have $f\notin C_2\setminus\{c\}$, in other words, $\overline{f}$ is on the row $s$. But now the double of the column 3 is the semistandard Young tableau with 2 columns:
$$
\begin{array}{lccc}
\hskip 0.5cm&E_2&E'_2&\hskip-0.2cm\\
&\overline{C'_{21}}&\overline{C_{21}}&\\
_s&\overline{y'}&\overline{f}&\\
\hline
&\overline{C_{22}}&\overline{C_{22}}&\\
\end{array}
$$
where we set as usual, $y=\delta_{(E_2\setminus\{f\})\cup(C_2\setminus\{c\})}(f)$, $E'_2=E_2\setminus\{f\}\cup\{y\}$, and we put $y'=\inf(C_{21}\cup\{y\})$, and $C'_{21}=C_{21}\setminus\{y'\}\cup\{y\}$ if $y'=c_{21}>y$, $C'_{21}=C_{21}$ otherwise.\\

In the next step, it is $\overline{y'}$ which will enter in column 1.\\

\noindent
\underbar{\bf Step 2} Let us prove: $y=y'=c$.\\

$\bullet$ Assume first $f=c$:

If $f=c$, $f\notin E_2$, so $y=f=c$ and since $y=c<c_{11}$, $y'=y=c$.\\

$\bullet$ Assume now $f<c$ and $f\in K_2$:

Since $K_2$ is the greatest subset with $k_2$ elements included in $[1,n]\setminus(A_2\cup D_2\cup J_2)=[1,n]\setminus(A_2\cup C_2)$, for all $t$, $t>f$ implies
$t\in K_2\cup A_2\cup C_2=E_2\cup C_2$. On the other hand $f\in K_2$ implies $f\in E_2$, and the the only $t>f$ not in $(E_2\setminus\{f\})\cup(C_2\setminus\{c\})$
is $c$, that is $y=\delta_{(E_2\setminus\{f\})\cup(C_2\setminus\{c\})}(f)=c$, and as above $y'=c$.\\

$\bullet$ Assume finally $f<c$ et $f\notin K_2$

Since $f$ is in $F_2=K_2\cup D_2$, $f\in D_2$. Since $c_{22}=\sup(C_{22})<f<c<c_{21}=\inf(C_{21})$, $f\notin C_2$. Then $f\in D_2\setminus C_2=I_2$. Especially, $f\in E_2$.

Write $I_2=\{x_1<\dots<x_i=f<\dots<x_r\}$. Set $J_2=\{y_1<\dots<y_r\}$. We have:
$$
y_i=\inf\{t\notin A_2\triangle D_2,~t>x_i=f, ~t>y_{i-1}\}\quad\text{ et }\quad y=\inf\{t\notin E_2\cup C_{21}\cup C_{22},~t>f\}.
$$
Therefore $c\notin  E_2\cup C_{21}\cup C_{22}$ et $c>f$, so $c\geq y$.

Assume, by contradiction, $c>y$.

In this case, $c_{22}<f=x_i<y<c$, so $y\notin C_2$. Since $y\notin E_2$, $y$ is not in $A_2$. If $y$ belonged to $D_2$, it would be in
$D_2\setminus A_2=C_2\setminus B_2$, which is impossible, so $y\notin A_2\cup D_2$, especially $y\notin A_2\triangle D_2$.

If $y>y_{i-1}$, then $y$ belongs to $\{t\notin A_2\triangle D_2,~t>x_i=f, ~t>y_{i-1}\}$, so $y\geq y_i$, but $y_i$ is in $C_2$, so $y>y_i$. And
$$
f=x_i<y_i<y<c.
$$
But $y_i\in C_2$ and there is no element in $C_2$ between $f$ and $c$, so this is impossible and $y\leq y_{i-1}$, we even have $y<y_{i-1}$ because $y\notin C_2$,
$$
x_{i-1}<x_i=f<y<y_{i-1}.
$$

Repeating the same argument, we prove that $x_1\leq f<y<y_1$, but
$$
y_1=\inf\{t\notin A_2\triangle D_2,~t>x_1\}
$$
and since $y\notin A_2\triangle D_2$, therefore $y\geq y_1$, which is impossible, so  $y=c$ and $y'=c$ as above.\\

\noindent
\underbar{\bf Step 3} End of the move.

By pushing $y'=c$, We finally obtain:
$$
\longmapsto~~\begin{array}{lccccc}
\hskip 0.5cm&E_1&B'_1&E'_2&B_2&\hskip-0.2cm\\
&\overline{C_{11}}&\overline{F_{11}}&\overline{C_{21}}&\overline{F_{21}}&\\
_s&\overline{x}&\overline{c}&\star&\star&\\
\hline
&\overline{C_{12}}&\overline{F_{12}}&\overline{C_{22}}&\overline{F_{22}}&\\
\end{array}.
$$
Which is the double of the orthogonal semistandard tableau:
$$
f(A_1,O_1,D_1\cup \{c\})g(B_2,O_2,(C_2\setminus\{c\})\cup\{\star\}).
$$

This is exactly the Move 2 case described after the theorem.\\

\

In all the considered cases, the entries situated under the star in the right, are still unchanged after the moving. So the condition $HS$ remains true in the tableau $dble(T)$, for all the columns coming after the column containing the star.\\

\subsection{Starting with a spin column}

\

In this section, we suppose the first column of the tableau $T$ is a spin column $\mathfrak f(A,D)$, with a trivial top (the entry in $(s,1)$ is $s$). Then we remove the $s$ top boxes in this column, put a $\star$ on the row $s$, and we look at the first move, under the hypothesis $HS$. If $s=n$, the first column is trivial and the jeu de taquin does not really consider the spin column. We suppose now $s<n$. 

We consider the tableau $dble(T)$, we remove the $s$ boxes in the top of the two first columns, we put two stars in $(s,1)$ and $(s,2)$. Remark that our hypothesis impose that the entry on the right of the star is unbarred. Indeed, if it was barred, in $dble(T)$, there is a $\overline{c}$ in the box $(s,3)$, hypothesis $HS$ implies that all the entries in the second column in $dble(T)$, below $s$ are barred, in other word, the spin column is $\mathfrak f([1,s],[s+1,n])$, and $\overline{c}\leq\overline{n}$, which is in contradiction with hypothesis $HS$.\\

Similarly, if it was 0, we saw there is no unbarred entry before $\star$ in column 2, that means the spin column is still $\mathfrak f([1,s],[s+1,n])$, the entry in the box $(s+1,3)$ is barred, it can only be $\overline{n}$, this implies $\overline{f}$ is larger than this entry, which is impossible.\\

The only possible case is $\boxed{\star|a}$. We just give the succession of tableaux (we do not put the index $sp$):
$$
\begin{array}{lccccc}
\hskip 0.5cm&&&E_{21}&B_{21}&\hskip 0.5cm\\
_s&\star&\star&e&b&\\
\hline
&s+1&A_{12}&E_{22}&B_{22}\\
&\begin{array}{c}\vdots\\n \end{array}&\overline{D_1}&\overline{C_2}&\overline{F_2}&
\end{array}~~~~\mapsto~~~~\begin{array}{lccccc}
\hskip 0.5cm&&&E_{21}&B_{21}&\hskip 0.5cm\\
_s&\star&x&\star&b&\\
\hline
&s+1&A_{12}&E_{22}&B_{22}\\
&\begin{array}{c}\vdots\\n \end{array}&\overline{D'_1}&\overline{C_2}&\overline{F_2}&
\end{array},
$$
with the same meaning for $x$ and $D'_1$, as above. By definition of $x=\gamma_{A_1\cup D_1}(e)$, we have $x=s$ (even if $e=s$). The next step is
$$
\mapsto~~~~\begin{array}{lccccc}
\hskip 0.5cm&&&E_{21}&B_{21}&\hskip 0.5cm\\
_s&s&\star&y&\star&\\
\hline
&s+1&A_{12}&E_{22}&B_{22}\\
&\begin{array}{c}\vdots\\n \end{array}&\overline{D'_1}&\overline{C'_2}&\overline{F_2}&
\end{array}~~~~\mapsto~~~~\begin{array}{lccccc}
\hskip 0.5cm&&&E_{21}&B_{21}&\hskip 0.5cm\\
_s&s&e&\star&\star&\\
\hline
&s+1&A_{12}&E_{22}&B_{22}\\
&\begin{array}{c}\vdots\\n \end{array}&\overline{D'_1}&\overline{C'_2}&\overline{F_2}&
\end{array}.
$$
We see that the two first columns is the double of the spin column
$$
\mathfrak f(\{1,\dots,s-1,e\},(D_1\setminus\{e\})\cup\{s\})
$$
from whose we remote the $s-1$ first boxes.\\

\subsection{Supression of the trivial top in the first column}

\

Suppose now $T$ is a semistandard tableau in $NQS_s^{[\bullet]}$, suppress the top $s$ boxes in the first column of $T$, getting a skew tableau $T\setminus S$. Apply the jeu de taquin to $T\setminus S$ (remark there is only one interior corner). Denote $T'\setminus S'$ the resulting tableau. Suppose $s>1$, $S'$ is a tableau with one column and $s-1$ boxes. Since the path of stars during two successive application of the jeu de taquin do not cross, if we apply the jeu de taquin to $T'\setminus S'$, the star will move every times horizontally. This is equivalent to say that the tableau $T'$ obtained when we fill up the empty boxes in $T'\setminus S'$ by $1,\dots,s-1$, is in $NQS_{s-1}^{[\bullet]}$. Remark we can also prove this point by using the preceding computation, and looking case by case the $HS$ condition along the new row $s-1$.\\

Then we can repeat the use of the jeu de taquin, getting a tableau $T''\setminus S''$, where the shape of $T''$ is the shape of $T$, where we suppress a column with height $s$ and add a new column with height $s-2$.

After $s$ repetition of the $ojdt$, we get a tableau $T^{(s)}$, with a shape smaller than the shape of $T$: the shape of $T^{(s)}$ if the shape of $T$ where we suppress a column with height $s$.\\  

\subsection{The map $ojdt^{-1}$ in the horizontal situation}

\

We proved that the orthogonal jeu de taquin, in the $HS$ situation is moving only horizontally. This operation is coming from a double action of the symplectic jeu de taquin on $dble(T)$. Let us now look for the inverse mapping. It is the composition of two inverse of the symplectic jeu de taquin on $dble(T)$. But we know how to define $(sjdt)^{-1}$ (see \cite{S}).\\

Let $T\setminus S$ be a skew orthogonal semistandard tableau. Denote $\sigma(T)$ the tableau obtained from $T$ by rotating $T$ half a tour, and replace each barred entry by the corresponding unbarred quantity, each unbarred by the corresponding barred quantity. Keep the 0.

\begin{exple}
Consider $\mathfrak{so}(9)$ ($n=4$) then
$$
T=\begin{tabular}{|c|c|c|} \hline
\raisebox{-2pt}{} & \raisebox{-2pt}{$1$} & \raisebox{-2pt}{$2$}\\
\hline
\raisebox{-2pt}{} & \raisebox{-2pt}{$3$} & \raisebox{-2pt}{$0$}\\
\hline
\raisebox{-2pt}{$0$} & \raisebox{-2pt}{$\overline{1}$}\\
\cline{1-2}
\raisebox{-2pt}{$\overline{1}$}\\
\cline{1-1}
\end{tabular}\\~~~~~\stackrel{\sigma}{\longmapsto} ~~~~~~ \sigma(T)=\begin{tabular}{cc|c|} \cline{3-3}
&&\raisebox{-2pt}{$~1$}\\
\cline{2-3}
&\multicolumn{1}{|c|}{\raisebox{-2pt}{$~1$}}&\raisebox{-2pt}{$0$}\\
\hline
\multicolumn{1}{|c|}{$0$}&\raisebox{-2pt}{$\overline{3}$}&\raisebox{-2pt}{}\\
\hline
\multicolumn{1}{|c|}{$\overline{2}$}&\raisebox{-2pt}{$\overline{1}$}&\raisebox{-2pt}{}\\
\hline
\end{tabular}.$$
\end{exple}

\begin{lem}

\

Defining similarly the $\sigma$ operation for a symplectic tableau, then we have,
$$
\sigma\left(dble(T)\right)=dble\left(\sigma(T)\right).
$$
\end{lem}

\begin{proof}
If $T$ has $r$ columns and the column $\mathcal C_j$ is $\mathcal C_j=f(A_j,O_j,D_j)$, then $\sigma(T)$ has $r$ columns and its $r+1-j$ column is
$$
\mathcal C'_{r+1-j}=f(D_j,O_j,A_j)
$$
Let us consider the double. The columns number $2j$, $2j+1$ in $dble(T)$ are:
$$
f(A_j\cup K_j,C_j)f(B_j,K_j\cup D_j).
$$

For the column $\mathcal C'_{r+1-j}$, we define the set $K'_{r+1-j}$, and immediately get $K'_{r+1-j}=K_j$. Therefore in $dble(\sigma(T))$, the columns $2(r-j)$, $2(r-j)+1$ are:
$$
f(D_j\cup K_j,B_j)f(C_j,A_j\cup K_j)
$$

This proves the lemma.\\
\end{proof}

Since, for symplectic tableaux, we have (see \cite{S})
$$
(sjdt)^{-1}=\sigma\circ sjdt\circ\sigma,
$$
and:
$$
dble\circ ojdt=sjdt\circ dble,
$$
then $ojdt$ is invertible and
$$
(ojdt)^{-1}=dble^{-1}\circ(sjdt)^{-1}\circ dble=dble^{-1}\circ\sigma\circ(sjdt)\circ \sigma \circ dble=\sigma\circ(ojdt)\circ \sigma.
$$

\begin{lem}

\

If $T$ satisfies the $HS$ condition, along the row $s$, then $T'=(ojdt)(T)$ is such that $\sigma(T')$ satisfies the $HS$ condition for the row $h+1-s$ if $h$ is the height of $T$.\\ 
\end{lem}

\begin{proof}

The preceding lemma proves that the motion of the star when we apply $ojdt$ at the row $h+1-s$ of $\sigma(T)$ would be always horizontal.\\

More precisely, if we look case by case the result of the elementary $ojdt$ move as described above, we directly verify that the tableau $\sigma\circ ojdt(T)$ satisfies the condition $HS$ on the row $h+1-s$. Let us for example look at one case:

Suppose that, in the row $s$ in $T$, there is $\boxed{\star|a|0}$, then, in the double, there is:
$$
\mapsto~~\begin{array}{lccccccc}
\hskip 0.5cm&E_{11}&B_{11}&E_{21}&B_2&E_3&\begin{array}{c}B_3\\ \overline{F_{31}}\end{array}&\hskip-0.2cm\\
_s&\star&\star&e_2&b_2&e_3&\overline{f_3}&\\
\hline
&\begin{array}{c}E_{12}\\ \overline{C_1}\end{array}&\begin{array}{c}B_{12}\\ \overline{F_1}\end{array}&\begin{array}{c}E_{22}\\ \overline{C_2}\end{array}&\overline{F_2}&\overline{C_3}&\overline{F_{32}}&
\end{array}.
$$
This gives successively:
$$
\mapsto~~\begin{array}{lccccccc}
\hskip 0.5cm&E_{11}&B_{11}&E_{21}&B_2&E_3&\begin{array}{c}B_3\\ \overline{F_{31}}\end{array}&\hskip-0.2cm\\
_s&x_2&e_2&\star&\star&e_3&\overline{f_3}&\\
\hline
&\begin{array}{c}E_{12}\\ \overline{C_1}\end{array}&\begin{array}{c}B_{12}\\ \overline{F'_1}\end{array}&\begin{array}{c}E_{22}\\ \overline{C'_2}\end{array}& \overline{F_2}&\overline{C_3}&\overline{F_{32}}&
\end{array},
$$
with $x_2=\gamma_{B_{11}\cup B_{12}\cup F_1}(e_2)\leq e_2$, and
$$
\mapsto~~\begin{array}{lccccccc}
\hskip 0.5cm&E_{11}&B_{11}&E_{21}&B_2&E'_3&\begin{array}{c}B_3\\ \overline{F_{31}}\end{array}&\hskip-0.2cm\\
_s&x_2&e_2&x_3&\overline{e_3}&\star&\star&\\
\hline
&\begin{array}{c}E_{12}\\ \overline{C_1}\end{array}&\begin{array}{c}B_{12}\\ \overline{F'_1}\end{array}&\begin{array}{c}E_{22}\\ \overline{C'_2}\end{array}&\overline{F'_2}&\overline{C_3}&\overline{F_{32}}&
\end{array},
$$
with $\inf{E_{22}}>x_3=\gamma_{B_2\cup F_2}(e_3)>\sup(B_2)$.

After the action of $\sigma$, we get:
$$
\mapsto~~\begin{array}{lccccccc}
\hskip 0.5cm&F_{32}&C_3&F'_2&\begin{array}{c}C'_2\\ \overline{E_{22}}\end{array}&\begin{array}{c}F'_1\\ \overline{B_{12}}\end{array}&\begin{array}{c}C_1\\ \overline{E_{12}}\end{array}&\hskip-0.2cm\\
_{h+1-s}&\star&\star&e_3&\overline{x_3}&\overline{e_2}&\overline{x_2}&\\
\hline
&\begin{array}{c}F_{31}\\ \overline{B_3}\end{array}&\overline{E'_3}&\overline{B_2}&\overline{E_{21}}&\overline{B_{11}}&\overline{E_{11}}&
\end{array}.
$$

We see that the $HS$ condition holds since:
$$
x_3>\sup(B_2),\quad e_2>\sup(E_{21}),\quad\text{and}\quad x_2\geq e_2>\sup(B_{11}).
$$
\end{proof}


\section{The map $p=(ojdt)^{max}$ is bijective}


\

We are now in position to define the map `$push$' or $p$ from $SS^{[\lambda]}$ to $\bigsqcup_{\mu\leq\lambda} QS^{[\mu]}$.\\

Let $T$ be an orthogonal semistandard tableau, with shape $\lambda$.

If $T$ is quasistandard, we put $p(T)=T$.

If $T$ is not quasistandard, we consider the greatest $s$ for which $T\in NQS_s^{[\lambda]}$ ({\sl i.e.} $T\notin NQS^{[\lambda]}_t$, for each $t>s$). The entry in the box $(s,1)$ in $T$ is $s$, then the two entries in the boxes $(s,1)$ and $(s,2)$ in $dble(T)$ are $s$ also, moreover, the $HS$ condition holds on the row $s$. We remove the $s$ top boxes in the first column of $T$, getting the skew tableau $T\setminus S$ and apply the $ojdt$ $s$ times to $T\setminus S$.\\

\begin{lem}

\

Let $\lambda'$ be the shape of $ojdt^s(T\setminus S)$, we get $\lambda'$ by removing a column with height $s$ and we have $ojdt(T)\notin NQS^{[\lambda']}_t$, for each $t>s$.\\
\end{lem}

\begin{proof}
After the horizontal action by $ojdt^s$ on $T$, we clearly get a tableau with the anounced shape. Now in \cite{AK} it is proved that the action of the $sjdt$ on a symplectic semistandard tableau which is in $NQS_s^{\langle \mu\rangle}$ but not in $NQS_t^{\langle \mu\rangle}$, for any $t>s$ is still not in 
$NQS_t^{\langle \mu'\rangle}$, for any $t>s$.

Applying two times this result on the tableau $dble(T)$, we get that $dble(ojdt(T))\notin NQS_t^{\langle 2\lambda'\rangle}$, for any $t>s$, that means the $HS$ condition does not hold on the row $t$, for any $t>s$, or $ojdt(T)\notin NQS_t^{[\lambda']}$, for any $t>s$.
\end{proof}

Replace now $T$ by $ojdt^s(T\setminus S)$ and repeat the above analyse. 

It is clear that, after a finite number of steps, this algorithm gives an orthogonal quasistandard tableau, with a shape smaller than the shape of $T$. We denote this tableau by:
$$
T'=p(T)=(ojdt)^{max}(T).
$$ 

Remark that if the first column of $T$ was a non trivial spin column, then the first column in $T'$ is a spin column. In all other cases, there is no spin column in $T'$.\\

The above algorithm defines the map $p:SS^{[\lambda]}\longrightarrow\bigsqcup_{\mu\leq\lambda}QS^{[\mu]}$. Let us look at the inverse map. We follow the method of \cite{AK}.

Start with an orthognal quasistandard tableau $U$, with shape $\mu$. Let $\lambda$ be a shape larger than $\mu$.\\

Apply $\sigma$ to $U$, seing it as a skew tableau inside a skew tableau with shape $\sigma(\lambda)$, put a $\star$ in the lowest inner corner in $\sigma(\lambda)\setminus\sigma(\mu)$, and apply the $ojdt$, getting a new skew tableau with shape $\sigma(\mu')$. Repeat this operation as far as there is inner corner in $\sigma(\lambda)\setminus\sigma(\mu')$.

At the end of this operation, get a skew tableau $T'$ with shape $\sigma(\lambda\setminus\mu)$, consider $T=\sigma(T')$ where the empty boxes are filled up in the top $\mu$ shape by a trivial tableau. By construction $T$ is orthogonal semistandard, with shape $\lambda$, and the above algorithm $(ojdt)^{max}$ apllied on $T$ is the inverse of this sequence of $\sigma\circ ojdt\circ\sigma$, that means $p(T)=U$, and $p$ is a bijective map.\\

Remark that if $U$ contains a spin column, this column becomes the last column in $\sigma(U)$, and at the end it becomes a spin column in $T$, as the first column of $T$.\\

\begin{thm}

\

The orthogonal jeu de taquin defines a bijection $p=(ojdt)^{max}$ from the set $SS^{[\lambda]}$ of orthogonal semistandard tableaux with shape $\lambda$ onto the disjoint union $\bigsqcup_{\mu\leq\lambda}QS^{[\mu]}$ of the set of orthogonal quasistandard tableaux with shape $\mu$ ($\mu\leq\lambda$).\\
\end{thm} 

Here is an example.

\begin{exple}

Suppose $n=5$, or $\mathfrak g=\mathfrak{so}(11)$
$$
T=\begin{tabular}{|c|c|c|} \hline
\raisebox{-2pt}{$1$} & \raisebox{-2pt}{$1$} & \raisebox{-2pt}{$2$}\\
\hline
\raisebox{-2pt}{$2$} & \raisebox{-2pt}{$3$} & \raisebox{-2pt}{$0$}\\
\hline
\raisebox{-2pt}{$0$} & \raisebox{-2pt}{$\overline{1}$}\\
\cline{1-2}
\raisebox{-2pt}{$\overline{2}$}\\
\cline{1-1}
\end{tabular}\\~~ \mapsto ~~dble(T)=\begin{tabular}{|c|c|c|c|c|c|} \hline
\raisebox{-2pt}{$1$} & \raisebox{-2pt}{$1$} & \raisebox{-2pt}{$1$}& \raisebox{-2pt}{$2$}& \raisebox{-2pt}{$2$}& \raisebox{-2pt}{$2$}\\
\hline
\raisebox{-2pt}{$2$} &\raisebox{-2pt}{$3$}& \raisebox{-2pt}{$3$}& \raisebox{-2pt}{$3$}& \raisebox{-2pt}{$5$} & \raisebox{-2pt}{$\overline{5}$}\\
\hline
\raisebox{-2pt}{$5$} &\raisebox{-2pt}{$\overline{5}$}&\raisebox{-2pt}{$\overline{2}$}& \raisebox{-2pt}{$\overline{1}$}\\
\cline{1-4}
\raisebox{-2pt}{$\overline{3}$}&\raisebox{-2pt}{$\overline{2}$}\\
\cline{1-2}
\end{tabular}.
$$
Then $T$ is in $NQS_2^{(0,1,1,1,0)}$. The $ojdt$ gives successively:
$$
\begin{tabular}{|c|c|c|} \hline
 & \raisebox{-2pt}{$1$} & \raisebox{-2pt}{$2$}\\
\hline
$\star$ & \raisebox{-2pt}{$3$} & \raisebox{-2pt}{$0$}\\
\hline
$0$ & \raisebox{-2pt}{$\overline{1}$}\\
\cline{1-2}
\raisebox{-2pt}{$\overline{2}$}\\
\cline{1-1}
\end{tabular}~~\mapsto \begin{tabular}{|c|c|c|} \hline
 & \raisebox{-2pt}{$1$} & \raisebox{-2pt}{$2$}\\
\hline
\raisebox{-2pt}{$3$} & \raisebox{-2pt}{$\star$} & \raisebox{-2pt}{$0$}\\
\hline
\raisebox{-2pt}{$0$}& \raisebox{-2pt}{$\overline {1}$}\\
\cline{1-2}
\raisebox{-2pt}{$\overline{2}$}\\
\cline{1-1}
\end{tabular}~~
\mapsto\begin{tabular}{|c|c|c|} \hline
 & \raisebox{-2pt}{$1$} & \raisebox{-2pt}{$2$}\\
\hline
\raisebox{-2pt}{$3$} & \raisebox{-2pt}{$0$} \\
\cline{1-2}
\raisebox{-2pt}{$0$}& \raisebox{-2pt}{$\overline {1}$}\\
\cline{1-2}
\raisebox{-2pt}{$\overline{2}$}\\
\cline{1-1}
\end{tabular}
~~
\mapsto\begin{tabular}{|c|c|c|} \hline
 $\star$& \raisebox{-2pt}{$1$} & \raisebox{-2pt}{$2$}\\
\hline
\raisebox{-2pt}{$3$} & \raisebox{-2pt}{$0$} \\
\cline{1-2}
\raisebox{-2pt}{$0$}& \raisebox{-2pt}{$\overline {1}$}\\
\cline{1-2}
\raisebox{-2pt}{$\overline{2}$}\\
\cline{1-1}
\end{tabular}
~~
\mapsto\begin{tabular}{|c|c|c|} \hline
 $\star$& \raisebox{-2pt}{$1$} & \raisebox{-2pt}{$2$}\\
\hline
\raisebox{-2pt}{$3$} & \raisebox{-2pt}{$0$} \\
\cline{1-2}
\raisebox{-2pt}{$0$}& \raisebox{-2pt}{$\overline {1}$}\\
\cline{1-2}
\raisebox{-2pt}{$\overline{2}$}\\
\cline{1-1}
\end{tabular}
~~
\mapsto\begin{tabular}{|c|c|c|} \hline
 \raisebox{-2pt}{$1$} & \raisebox{-2pt}{$2$} \\
\cline{1-2}
\raisebox{-2pt}{$3$} & \raisebox{-2pt}{$0$} \\
\cline{1-2}
\raisebox{-2pt}{$0$}& \raisebox{-2pt}{$\overline {1}$}\\
\cline{1-2}
\raisebox{-2pt}{$\overline{2}$}\\
\cline{1-1}
\end{tabular}=U
$$
with $U\in QS^{(0,0,1,1,0)}$.\\

Conversely:
$$\aligned
&U=\begin{tabular}{|c|c|} \hline
\raisebox{-2pt}{$~1$}&\raisebox{-2pt}{$~2$}\\
\hline
\raisebox{-2pt}{$3$}&\raisebox{-2pt}{$0$}\\
\hline
\raisebox{-2pt}{$0$}&\raisebox{-2pt}{$\overline{1}$}\\
\hline
\raisebox{-2pt}{$\overline{2}$}\\
\cline{1-1}
\end{tabular}~~\stackrel{\sigma}{\mapsto} \begin{tabular}{cc|c|} \cline{3-3}
&&\raisebox{-2pt}{$~2$}\\
\cline{2-3}
&\multicolumn{1}{|c|}{\raisebox{-2pt}{$~1$}}&\raisebox{-2pt}{$0$}\\
\hline
\multicolumn{1}{|c|}{$\star_2$}&\raisebox{-2pt}{$0$}&\raisebox{-2pt}{$\overline{3}$}\\
\hline
\multicolumn{1}{|c|}{$\star_1$}&\raisebox{-2pt}{$\overline{2}$}&\raisebox{-2pt}{$\overline{1}$}\\
\hline
\end{tabular}~~\mapsto  \begin{tabular}{cc|c|} \cline{3-3}
&&\raisebox{-2pt}{$~2$}\\
\cline{2-3}
&\multicolumn{1}{|c|}{\raisebox{-2pt}{$~1$}}&\raisebox{-2pt}{$0$}\\
\hline
\multicolumn{1}{|c|}{}&\raisebox{-2pt}{$0$}&\raisebox{-2pt}{$\overline{3}$}\\
\hline
\multicolumn{1}{|c|}{$~\overline{2}$}&\raisebox{-2pt}{$\star_1$}&\raisebox{-2pt}{$\overline{1}$}\\
\hline
\end{tabular}~~\mapsto  \begin{tabular}{cc|c|} \cline{3-3}
&&\raisebox{-2pt}{$~2$}\\
\cline{2-3}
&\multicolumn{1}{|c|}{\raisebox{-2pt}{$~1$}}&\raisebox{-2pt}{$0$}\\
\hline
\multicolumn{1}{|c|}{}&\raisebox{-2pt}{$0$}&\raisebox{-2pt}{$\overline{3}$}\\
\hline
\multicolumn{1}{|c|}{$~\overline{2}$}&\raisebox{-2pt}{$\overline{1}$}\\
\cline{1-2}
\end{tabular}\\
&~~\mapsto  \begin{tabular}{cc|c|} \cline{3-3}
&&\raisebox{-2pt}{$~2$}\\
\cline{2-3}
&\multicolumn{1}{|c|}{\raisebox{-2pt}{$~1$}}&\raisebox{-2pt}{$0$}\\
\hline
\multicolumn{1}{|c|}{$\star_2$}&\raisebox{-2pt}{$0$}&\raisebox{-2pt}{$\overline{3}$}\\
\hline
\multicolumn{1}{|c|}{$~\overline{2}$}&\raisebox{-2pt}{$\overline{1}$}\\
\cline{1-2}
\end{tabular}~~\mapsto  \begin{tabular}{cc|c|} \cline{3-3}
&&\raisebox{-2pt}{$~2$}\\
\cline{2-3}
&\multicolumn{1}{|c|}{\raisebox{-2pt}{$~1$}}&\raisebox{-2pt}{$0$}\\
\hline
\multicolumn{1}{|c|}{$0$}&\raisebox{-2pt}{$\star_2$}&\raisebox{-2pt}{$\overline{3}$}\\
\hline
\multicolumn{1}{|c|}{$~\overline{2}$}&\raisebox{-2pt}{$\overline{1}$}\\
\cline{1-2}
\end{tabular}~~\mapsto  \begin{tabular}{cc|c|} \cline{3-3}
&&\raisebox{-2pt}{$~2$}\\
\cline{2-3}
&\multicolumn{1}{|c|}{\raisebox{-2pt}{$~1$}}&\raisebox{-2pt}{$0$}\\
\hline
\multicolumn{1}{|c|}{$0$}&\raisebox{-2pt}{$\overline{3}$}&\raisebox{-2pt}{$\star_2$}\\
\hline
\multicolumn{1}{|c|}{$~\overline{2}$}&\raisebox{-2pt}{$\overline{1}$}\\
\cline{1-2}
\end{tabular}\mapsto  \begin{tabular}{cc|c|} \cline{3-3}
&&\raisebox{-2pt}{$~2$}\\
\cline{2-3}
&\multicolumn{1}{|c|}{\raisebox{-2pt}{$~1$}}&\raisebox{-2pt}{$0$}\\
\hline
\multicolumn{1}{|c|}{$0$}&\raisebox{-2pt}{$\overline{3}$}\\
\cline{1-2}
\multicolumn{1}{|c|}{$~\overline{2}$}&\raisebox{-2pt}{$\overline{1}$}\\
\cline{1-2}
\end{tabular} \\
&~~\stackrel{\sigma}{\mapsto} \begin{tabular}{c|c|c|}
\cline{2-3}
&\raisebox{-2pt}{$1$}&\raisebox{-2pt}{$2$}\\
\cline{2-3}   
&\raisebox{-2pt}{$3$}&\raisebox{-2pt}{$0$}\\
\cline{1-3}
\multicolumn{1}{|c|}{\raisebox{-2pt}{$0$}}&\raisebox{-2pt}{$\overline{1}$}\\
\cline{1-2}
\multicolumn{1}{|c|}{\raisebox{-2pt}{$\overline{2}$}}\\
\cline{1-1}
\end{tabular}~~\mapsto \begin{tabular}{|c|c|c|} \hline\raisebox{-2pt}{$1$}&\raisebox{-2pt}{$1$}&\raisebox{-2pt}{$2$}\\
\hline
\raisebox{-2pt}{$2$}&\raisebox{-2pt}{$3$}&\raisebox{-2pt}{$0$}\\
\hline\raisebox{-2pt}{$0$}&\raisebox{-2pt}{$\overline{1}$}\\
\cline{1-2}\raisebox{-2pt}{$\overline{2}$}\\
\cline{1-1}
\end{tabular}=T
\endaligned
$$
\end{exple}





\end{document}